\documentclass[12pt]{article}
\usepackage{amsmath}
\usepackage{graphicx}
\usepackage{amsfonts}
\usepackage{verbatim}
\usepackage{latexsym}
\usepackage{epstopdf}
\usepackage{color}

\usepackage{a4}
\makeatletter
\@addtoreset{equation}{section}

\begin{document}

\newcommand{\EE}{\mathbb{E}}
\newcommand{\PP}{\mathbb{P}}
\newcommand{\RR}{\mathbb{R}}
\newcommand{\SM}{\mathbb{S}}
\newcommand{\ZZ}{\mathbb{Z}}
\newcommand{\ind}{\mathbf{1}}
\newcommand{\LL}{\mathbb{L}}
\def\F{{\cal F}}
\def\G{{\cal G}}
\def\P{{\cal P}}

\newtheorem{theorem}{Theorem}[section]
\newtheorem{lemma}[theorem]{Lemma}
\newtheorem{coro}[theorem]{Corollary}
\newtheorem{defn}[theorem]{Definition}
\newtheorem{assp}[theorem]{Assumption}
\newtheorem{cond}[theorem]{Condition}
\newtheorem{expl}[theorem]{Example}
\newtheorem{prop}[theorem]{Proposition}
\newtheorem{rmk}[theorem]{Remark}
\newtheorem{conj}[theorem]{Conjecture}

\newcommand\tq{{\scriptstyle{3\over 4 }\scriptstyle}}
\newcommand\qua{{\scriptstyle{1\over 4 }\scriptstyle}}
\newcommand\hf{{\textstyle{1\over 2 }\displaystyle}}
\newcommand\hhf{{\scriptstyle{1\over 2 }\scriptstyle}}
\newcommand\hei{\tfrac{1}{8}}

\newcommand{\eproof}{\indent\vrule height6pt width4pt depth1pt\hfil\par\medbreak}

\def\a{\alpha}
\def\e{\varepsilon} \def\z{\zeta} \def\y{\eta} \def\o{\theta}
\def\vo{\vartheta} \def\k{\kappa} \def\l{\lambda} \def\m{\mu} \def\n{\nu}
\def\x{\xi}  \def\r{\rho} \def\s{\sigma}
\def\p{\phi} \def\f{\varphi}   \def\w{\omega}
\def\q{\surd} \def\i{\bot} \def\h{\forall} \def\j{\emptyset}

\def\be{\beta} \def\de{\delta} \def\up{\upsilon} \def\eq{\equiv}
\def\ve{\vee} \def\we{\wedge}

\def\D{\Delta} \def\O{\Theta} \def\L{\Lambda}
\def\X{\Xi} \def\Si{\Sigma} \def\W{\Omega}
\def\M{\partial} \def\N{\nabla} \def\Ex{\exists} \def\K{\times}
\def\V{\bigvee} \def\U{\bigwedge}

\def\1{\oslash} \def\2{\oplus} \def\3{\otimes} \def\4{\ominus}
\def\5{\circ} \def\6{\odot} \def\7{\backslash} \def\8{\infty}
\def\9{\bigcap} \def\0{\bigcup} \def\+{\pm} \def\-{\mp}
\def\la{\langle} \def\ra{\rangle}

\def\proof{\noindent{\it Proof. }}
\def\tl{\tilde}
\def\trace{\hbox{\rm trace}}
\def\diag{\hbox{\rm diag}}
\def\for{\quad\hbox{for }}
\def\refer{\hangindent=0.3in\hangafter=1}

\newcommand\wD{\widehat{\D}}
\newcommand{\ka}{\kappa_{10}}

\title{
\bf Stationary distribution of the stochastic theta method for nonlinear stochastic differential equations}

\author
{{\bf  Yanan Jiang, Wei Liu, Lihui Weng\footnote{Corresponding author, Email: 1151022187@qq.com}}
\\
Department of Mathematics, \\ Shanghai Normal University, Shanghai, China
}

\date{}

\maketitle

\begin{abstract}
The existence and uniqueness of the stationary distribution of the numerical solution generated by the stochastic theta method is studied. When the parameter $\theta$ takes different values, the requirements on the drift and diffusion coefficients are different. The convergence of the numerical stationary distribution to the true counterpart is investigated. Several numerical experiments are presented to demonstrate the theoretical results.

 numerical stationary distribution.
\end{abstract}

\section{Introduction} \label{secintro}

The classical method to find the stationary distributions of some stochastic differential equations (SDEs) is to solve the corresponding Kolmogorov-Fokker-Planck equations. However, it is not trivial to find the solution to those partial differential equations when some nonlinearity appears in the drift or the diffusion coefficient of the SDEs. In this paper, the alternative path that the stationary distributions generated by some numerical methods for SDEs are used as the approximates to those of the underlying equations is investigated.
\par
In the series papers \cite{MYY2005a,YM2004a,YM2005a}, the authors studied the approximates to stationary distributions of SDEs and SDEs with Markovian switching by using the Euler-Maruyama method. In \cite{BSY2016}, the approximations of invariant measures of SDEs with different sorts of Markovian switchings were investigated using the Euler-Maruyama method. Both of the drift and diffusion coefficients of the SDEs in those papers above need to satisfy the global Lipschitz condition. As indicated in \cite{HJK2011a}, the classical Euler-Maruyama fails to convergence when either the drift or the diffusion coefficient grows super-linearly. To tackle this drawback, the backward Euler-Maruyama method was employed in \cite{LM2015} for those SDEs with the super-linear drift coefficient. Higher order methods were also discussed for sampling the invariant measures \cite{abdulle2017optimal,talay1990second}.
\par
In this paper, we study the numerical stationary distributions of the stochastic theta (ST) method and discuss the effect of the choice of the theta on the conditions of the coefficients. Different types of asymptotic properties of the stochastic theta method for SDEs have been widely investigated. In \cite{Hig2000a}, the stability of the ST method was studied in both the mean-square and almost sure senses. The stability of the ST method with nonrandom variable step sizes for bilinear, nonautonomous, homogenous test equations was investigated in \cite{Rod05}. The ST method was applied to a test system with stabilising and destabilising stochastic perturbations and almost sure asymptotic stability was analysed in \cite{BBKR2012a}. The abilities to preserve the almost sure and the mean square exponential stabilities were discussed for different choices of the theta in \cite{CW2012} and \cite{ZW2014}, respectively. The asymptotic boundedness of the ST method was studied in \cite{QLH2014}. The results presented in this paper could be regarded as a complement to the existing study of the asymptotic behaviours of the ST method.
\par
This paper is constructed in the following way. The necessary mathematical preliminaries are presented in Section \ref{secmathpre}. Section \ref{secmain} contains the main results. Some numerical examples are used to illustrate the theoretical results in Section \ref{secsimu}. We summarize the paper by Section \ref{secconclu}.

\section{Mathematical Preliminaries} \label{secmathpre}
In this paper, let $(\Omega , \F, \PP)$ be a complete probability space with a filtration $\left\{\F_t\right\}_{t \ge 0}$ satisfying the usual conditions that it is right continuous and increasing while $\F_0$ contains all $\PP$-null sets. Let $|\cdot|$ denote the Euclidean norm in $\RR^d$. The transpose of a vector or matrix, $M$, is denoted by $M^T$ and the trace norm of a matrix, $M$, is denoted by $|M| = \sqrt{\trace(M^T M)}$.
\par
Let $f,g : \RR^d \rightarrow \RR^d$. To keep symbols simple, let $B(t)$ be a scalar Brownian motion. The results in this paper can be extended to the case of multi-dimensional Brownian motions.
\par
We consider the $d$-dimensional stochastic differential equation of the It\^o type
\begin{equation}
 \label{SDE}
dx(t) = f(x(t))dt + g(x(t))dB(t)
\end{equation}
with initial value $x(0) = x_0$.
\par
Now, we present the conditions on the coefficients.

\begin{cond}
\label{ffgg}
 Assume there exists a constant $K_1>0$ such that for any $x,y \in \RR^d$
\begin{equation*}
 |f(x)-f(y)|^2\vee|g(x)-g(y)|^2\leq K_{1} |x-y|^2.
\end{equation*}
\end{cond}
\begin{cond}
\label{xyff}
 Assume there exists a constant $K_{2}<0$ such that for any $x,y \in \RR^d$
\begin{equation*}
 \langle x-y, f(x) - f(y)  \rangle \leq K_{2} |x-y|^2.
\end{equation*}
\end{cond}
In addition, we require that
\begin{equation}
\label{k1k2}
2K_{2}+K_{1}<0.
\end{equation}
The next two conditions can be derived from Conditions \ref{ffgg} and \ref{xyff} but with a little bit complicated coefficients. For the simplicity, we give two new conditions as follows.
\begin{cond}\label{monocond}
 \label{xf}
There exist constants $\mu< 0$ and $a>0$ such that for any $x \in \RR^d$
\begin{equation*}
 \langle x,f(x)\rangle \leq \mu|x|^2+a.
\end{equation*}
\end{cond}

\begin{cond}
\label{ligrfg}
There exist positive constants $\sigma$, $\kappa$, $b$ and $c$ such that
for any $x\in\mathbb{R}^d$
\begin{equation}
\label{gg}
 |g(x)|^2 \leq \sigma|x|^2+b,
\end{equation}
and
\begin{equation}
\label{ff}
 |f(x)|^2 \leq \kappa|x|^2+c,
\end{equation}
\end{cond}
In addition, we require that
\begin{equation}
\label{ms}
2\mu+\sigma<0.
\end{equation}
\par \noindent
The existence and uniqueness of the underlying SDE (\ref{SDE}) has been broadly studied. We refer the readers to Theorem 3.1 in \cite{YM2003a} for a quite general theory. There are other more general theories, the reason we refer the readers to this one is that the structure of it is similar to the following theory, Theorem \ref{mainthm}.
\par \noindent
The stochastic theta method to SDE (\ref{SDE}) is defined by
\begin{equation}
 \label{STM}
X_{k+1} = X_k + \theta f(X_{k+1})h +(1-\theta) f(X_{k})h+ g(X_k) \D B_k,~~~X_0 = x(0)= x_0,
\end{equation}
where $\D B_k = B(t_{k+1}) - B(t_k)$ is the Brownian motion increment and $t_k = kh$, for $k = 1,2,..$.
\par
The proof of the next lemma is similar to those in \cite{LM2015,MS2013a}.
\begin{lemma}
\label{pre1}
Let Condition \ref{ffgg} \ref{xyff} hold and $\theta hK_2<1$, the ST method \eqref{STM} is well defined.
\end{lemma}
\begin{proof}
It is useful to write \eqref{STM} as
\begin{equation*}
X_{k+1}-\theta f(X_{k+1})h  = X_k +(1-\theta) f(X_{k})h+ g(X_k) \D B_k.
\end{equation*}
Define a function $G:\mathbb{R}^d \rightarrow\mathbb{R}^d$ by $G(x)=x-f(x)\theta h$. Since
\begin{align*}
\langle x-y,G(x)-G(y)\rangle&\geq\langle x-y,x-y-\theta h(f(x)-f(y))\rangle\\
  &\geq|x-y|^2-\theta hK_2|x-y|^2\\
  &=(1-\theta hK_2)|x-y|^2>0,
\end{align*}
for  $\theta hK_2<1$, we know that $G$ has the
 inverse function $G^{-1}:\mathbb{R}^d \rightarrow\mathbb{R}^d$. And $G(x)$ is monotone. The ST method \eqref{STM} can be written as
\begin{equation}
\label{STM2}
X_{k+1} = G^{-1}(X_k +(1-\theta) f(X_{k})h+ g(X_k) \D B_k).
\end{equation}
Thus, the ST method \eqref{STM2} is well defined. \eproof
\end{proof}
\begin{lemma}
\label{pre2}
Let Conditions \ref{ffgg} to \ref{ligrfg} hold, then
\begin{equation}
\mathbb{P}(X_{k+1}\in B|X_k=x)=\mathbb{P}(X_1\in B|X_0=x),
\end{equation}
for any Borel set $B\subset\mathbb{R}^d$.
\end{lemma}
\begin{proof}
 If $X_k = x$ and $X_0 = x$, by (\ref{STM}) we see
\begin{equation*}
 X_{k+1} - \theta f(X_{k+1})h = x + (1-\theta)f(x)h+ g(x) \D B_k,
\end{equation*}
and
\begin{equation*}
 X_1 - \theta f(X_1)h = x + (1-\theta)f(x)h+ g(x) \D B_0.
\end{equation*}
Because $\D B_k$ and $\D B_0$ are identical in probability law, comparing the two equations above, we know that $X_{k+1} - \theta f(X_{k+1})h$ and $X_1 - \theta f(X_1)h$
have the identical probability law. Then, due to Lemma \ref{pre1}, we have that $X_{k+1}$ and $X_1$ are identical in probability law under $X_k = x$ and $X_0 = x$.
Therefore, the assertion holds.   \eproof
\end{proof}

To prove Theorem \ref{hMp}, we cite the following classical result (see, for example, Lemma 9.2 on page 87 of \cite{M2008a}).
\begin{lemma}
\label{clas}
 Let $h(x,\omega)$ be a scalar bounded measurable random function of $x$, independent of $\F_s$. Let $\zeta$ be an $\F_s$-measurable random variable. Then
\begin{equation*}
 \EE (h(\zeta,\omega) \big\vert \F_s) = H(\zeta),
\end{equation*}
where $H(x) = \EE h(x,\omega)$.
\end{lemma}
For any $x \in \RR^d$ and any Borel set $B \subset \RR^d$, define
\begin{equation*}
 \PP(x,B) := \PP (X_1 \in B \big\vert X_0 = x)~\text{and}~\PP_k(x,B) := \PP (X_k \in B \big\vert X_0 = x).
\end{equation*}

\begin{theorem}
 \label{hMp}
The solution generated by the ST method (\ref{STM}) is a homogeneous Markov process with transition probability kernel $\PP(x,B)$.
\end{theorem}
\begin{proof}
 The homogeneous property follows Lemma \ref{pre2}, so we only need to show the Markov property. Define
\begin{equation*}
 Y_{k+1}^x = G^{-1} (x + (1 - \theta)f(x)h + g(x) \D B_k),
\end{equation*}
for $x \in \RR^d$ and $k \geq 0$. By (\ref{STM2}) we know that $X_{k+1} = Y_{k+1}^{X_k}$. Let $\G_{t_{k+1}} = \sigma \{ B(t_{k+1}) - B(t_k) \}$.
Clearly, $\G_{t_{k+1}}$ is independent of $\F_{t_k}$. Moreover, $Y_{k+1}^x$ depends completely on the increment $B(t_{k+1}) - B(t_k)$, so is $\G_{t_{k+1}}$-measurable.
Hence, $Y_{k+1}^x$ is independent of $\F_{t_k}$. Applying Lemma \ref{clas} with $h(x,\omega) = I_B(Y_{k+1}^x)$, we compute that
\begin{align*}
 &\PP (X_{k+1} \in B \big\vert \F_{t_k}) = \EE(I_B(X_{k+1}) \big\vert \F_{t_k}) = \EE \left( I_B(Y_{k+1}^{X_k}) \big\vert \F_{t_k} \right)
= \EE \left( I_B(Y_{k+1}^x)\right) \big\vert_{x=X_k} \\
&~= \PP(x,B) \big\vert_{x=X_k} = \PP(X_k,B) = \PP (X_{k+1} \in B \big\vert X_k).
\end{align*}
The proof is complete.   \eproof
\end{proof}
Therefore, we see that $\PP(\cdot,\cdot)$ is the one-step transition probability and $\PP_k(\cdot,\cdot)$ is the $k$-step transition probability, both of which are induced by the BEM solution.
\par
We state a simple version of the discrete-type Gronwall inequality in the next Lemma (see, for example, \cite{M1991a}).
\begin{lemma}\label{disGronine}
Let $\{u_n\}$ and $\{w_n\}$ be nonnegative sequences, and $\alpha$ be a nonnegative constant. If
\begin{equation*}
u_n \leq \alpha + \sum_{k=0}^{n-1} u_k w_k~~~\text{for}~n\geq 0,
\end{equation*}
then
\begin{equation*}
u_n \leq \alpha \exp\left(\sum_{k=0}^{n-1}w_k\right).
\end{equation*}
\end{lemma}
Denote the family of all probability measures on $\RR^d$ by $\P(\RR^d)$. Define by $\LL$ the family of mappings $F: \RR^d \rightarrow \RR$ satisfying
\begin{equation*}
 |F(x) - F(y)| \leq |x-y| ~~~ \text{and} ~~~ |F(x)| \leq 1,
\end{equation*}
for any $x,y \in \RR^d$. For $\PP_1,\PP_2 \in \P(\RR^d)$, define metric $d_{\LL}$ by
\begin{equation*}
 d_{\LL}(\PP_1,\PP_2) = \sup_{F \in \LL} \left| \int_{\RR^d} F(x)\PP_1(dx) -  \int_{\RR^d} F(x)\PP_2(dx) \right|.
\end{equation*}
The weak convergence of probability measures can be illustrated in terms of metric $d_{\LL}$ \cite{IW1981a}. That is,
a sequence of probability measures $\{\PP_k\}_{k \geq 1}$ in $\P(\RR^d)$ converge weakly to a probability measure $\PP \in \P(\RR^d)$ if and only if
\begin{equation*}
 \lim_{k \rightarrow \infty} d_{\LL}(\PP_k,\PP) = 0.
\end{equation*}
Then we define the stationary distribution for $\{X_k\}_{k \geq 0}$ by using the concept of weak convergence.
\begin{defn}
 For any initial value $x \in \RR^d$ and a given step size $\D t > 0$, $\{X_k\}_{k \geq 0}$ is said to have a stationary distribution $\Pi_{\D t} \in \P(\RR^d)$ if
the $k$-step transition probability measure $\PP_k(x,\cdot)$ converges weakly to $\Pi_{\D t}(\cdot)$ as $k \rightarrow \infty$ for every $x \in \RR^d$, that is
\begin{equation*}
 \lim_{k \rightarrow \infty} \left( \sup_{F \in \LL} \left| \EE (F(X_k)) - E_{\Pi_{\D t}} (F) \right| \right) = 0,
\end{equation*}
where
\begin{equation*}
 E_{\Pi_{\D t}} (F) = \int_{\RR^d} F(y) \Pi_{\D t}(dy).
\end{equation*}
\end{defn}
\par \noindent
In \cite{YM2005a}, the authors presented a very general theory, Theorem 3.1, on the existence and uniqueness of the stationary distribution for any one step numerical methods. We adapt it here and state the theory for the stochastic theta method as follows.
\begin{theorem}
\label{mainthm}
Assume that the following three requirements are fulfilled.
\begin{itemize}
\item
For any $\e >0$ and $x_0 \in \RR^d$, there exists a constant $R = R(\e,x_0) > 0$ such that
\begin{equation}\label{ass1}
 \PP (|X_k^{x_0}| \geq R) < \e,~~~\text{for any}~k \geq 0.
\end{equation}

\item
 For any $\e >0$ and any compact subset $K$ of $\RR^d$, there exists a positive integer $k^* = k^*(\e,K)$ such that
\begin{equation}\label{ass2}
 \PP(|X_k^{x_0} - X_k^{y_0}|<\e) \geq 1 - \e,~~~\text{for any}~ k \geq k^* ~\text{and any}~(x_0,y_0) \in K \K K.
\end{equation}

\item
 For any $\e >0$, $n \geq 1$ and any compact subset $K$ of $\RR^d$, there exists a $R = R(\e,n,K) > 0$ such that
\begin{equation}\label{ass3}
 \PP \left(\sup_{0 \leq k \leq n} |X_k^{x_0}| \leq R\right) > 1 - \e, ~~~ \text{for any}~x_0 \in K.
\end{equation}
\end{itemize}
Then the numerical solution generated by the stochastic theta method $\{X_k\}_{k \geq 0}$ has a unique stationary distribution $\Pi_{\D t}$.
\end{theorem}
\begin{rmk}
Although the theory is very general, the conditions in it are in the sense of probability which are not easy to check. In this paper, we give some coefficients related conditions, i.e Conditions \ref{ffgg} to \ref{ligrfg}, and prove the existence and uniqueness of the stationary distribution of the solution generated by the ST method under those conditions.
\end{rmk}

\section{Main Results}\label{secmain}

In this section, we present the main results of this paper. Since different choices of the parameter $\theta$ in (\ref{STM}) require different requirements on the coefficients, $f$ and $g$, we divide this section into three parts. We discuss the case when $\theta \in [0,1/2)$ in Section \ref{thetalessthanahalf} and the situation when $\theta \in [1/2,1]$ is presented in Section \ref{thetabiggerthanahalf}. The convergence of the numerical stationary distribution to the underlying counterpart is discussed in Section \ref{theconvergencesec}.
\par

\subsection{$\theta \in [0,1/2)$} \label{thetalessthanahalf}

\begin{lemma}
\label{lemma11}
 Assume Conditions \ref{xf} and \ref{ligrfg} hold, then for
 $h<-(2\mu+\sigma)/(1-\theta)^2\kappa$, the solution generated by the ST method \eqref{STM} obeys
\end{lemma}
 \begin{equation*}
 \EE|X_{k}|^{2}\leq C_{1},
 \end{equation*}
 where $C_{1}$ is a constant that does not rely on k.\\
\begin{proof}
Applying Conditions \ref{xf} and \ref{ligrfg}, we have
\begin{align*}
|X_{k+1}|^2=&\langle X_{k+1},X_{k}+(1-\theta)f(X_{k})h+g(X_{k})\Delta B_{k}\rangle+\theta h\langle X_{k+1},F(X_{k+1})\rangle \\
\leq & \frac{1}{2}|X_{k+1}|^{2}+\frac{1}{2}[X_{k}+(1-\theta)hf(X_{k})+g(X_{k})\Delta B_{K}]^{2}+\theta h \mu|X_{k+1}|^2+a\theta h\\
=&\frac{1}{2}|X_{k+1}|^2+\frac{1}{2}[|X_{k}|^2+(1-\theta)^2 f(X_{k})^2 h^2+g(X_{K})^{2} \Delta B_{K}^2+2(1-\theta)h X_{k}f(X_{k})]\\
&+\theta h\mu|X_{k+1}|^2+a\theta h+A_{1},
\end{align*}
where $A_{1}=\langle X_{k},g(X_{k})\rangle \Delta B_{k}+2(1-\theta)h\langle f(X_{k}),g(X_{k})\rangle\Delta B_{k}$.
Since $\mathbb{E}\Delta B_{k}=0$, we have $\mathbb{E}A_{1}=0$. By iteration, we have
\begin{align*}
\mathbb{E}|X_{k+1}|^2& \leq A_{2}\mathbb{E}|X_{k}|^2+A_{3}\\
&\leq A_{2}(A_{2}\mathbb{E}|X_{k-1}|^2+A_{3})+A_{3}\\
&\leq A_{2}^{k+1}\mathbb{E}|X_{0}|^2+A_{3}+A_{2}A_{3}+A_{2}^2A_{3}+\cdots+A_{2}^k A_{3},
\end{align*}
where
\begin{equation*}
A_{2}=\frac{1+(1-\theta)^2 h^2 \kappa+\Delta B_{k}^2 \sigma+2(1-\theta)h\mu}{1-2\mu\theta h},
\end{equation*}
and
\begin{equation*}
A_{3}=\frac{(1-\theta)^2 h^2c+bh+2(1-\theta)ha+2a\theta h}{1-2\mu\theta h}.
\end{equation*}
Due to the facts $h<-(2\mu+\sigma)/(1-\theta)^2\kappa$, $1-2\mu\theta h>0$ and $\kappa>0$ and \eqref{ms},we have $0<A_{2}<1$. This complete the proof.     \eproof
\end{proof}

\begin{lemma}
\label{lemma12}
 Let Condition \ref{ffgg} and \ref{xyff} hold. Then, for $h<-(2K_2+K_1)/(1-\theta)^2K_1$ and any two initial values $x,y\in\mathbb{R}^{d}$ with $x \neq y$ the solutions generated by the ST method \eqref{STM} satisfy
 \begin{equation*}
 \EE|X_{k}^x-X_{k}^y|^{2}\leq C_{3}\mathbb{E}|x-y|^2,
 \end{equation*}
 where
\begin{equation*}
C_3=\left[\frac{1+(1-\theta)^2h^2K_{1}+hK_{1}+2K_{2}(1-\theta)h}{1-2K_{2}\theta h}\right]^{k+1},
\end{equation*}
and $\lim\limits_{k\rightarrow+\infty}C_{3}=0$.
\end{lemma}
\begin{proof}
From \eqref{STM}, we have
\begin{align*}
|X_{k+1}^x-X_{k+1}^y|=&X_{k}^x-X_{k}^y+\theta[f(X_{k+1}^x)-f(X_{k+1}^y)]h+(1-\theta)[f(X_{k}^x)-f(X_{k}^y)]h\\
&+[g(X_{k}^x)-g(X_{k}^y)]\Delta B_{k}.
\end{align*}
Applying Conditions \ref{ffgg} and \ref{xyff}, we have
\begin{align*}
|X_{k+1}^x-X_{k+1}^y|^2=&\langle X_{k+1}^x-X_{k+1}^y,X_{k}^x-X_{k}^y+(1-\theta)[f(X_{k}^x)-f(X_{k}^y)]h\\
&+[g(X_{k}^x)-g(X_{k}^y)]\Delta B_{k}\rangle+\langle X_{k+1}^x-X_{k+1}^y,f(X_{k+1}^x)-f(X_{k+1}^y)\rangle\theta h\\
\leq&\frac{1}{2}|X_{k+1}^x-X_{k+1}^y|^2+\frac{1}{2}[X_{k}^x-X_{k}^y+(1-\theta)[f(X_{k}^x)-f(X_{k}^y)]h\\
&+[g(X_{k}^x)-g(X_{k}^y)]\Delta B_{k}]^2+\theta hK_{2}|X_{k+1}^x-X_{k+1}^y|^2.
\end{align*}
Then, we have
\begin{align*}
(\frac{1}{2}-K_{2}\theta h )|X_{k+1}^x-X_{k+1}^y|^2\leq &\frac{1}{2}|X_{k}^x-X_{k}^y|^2+\frac{1}{2}(1-\theta)^2[f(X_{k}^x)-f(X_{k}^y)]^2h^2\\
&+\frac{1}{2}[g(X_{k}^x)-g(X_{k}^y)]^2\Delta B_{k}^2\\
&+\langle X_{k}^x-X_{k}^y,f(X_{k}^x)-f(X_{k}^y)\rangle (1-\theta)h+Q_{1},
\end{align*}
where
\begin{equation*}
Q_{1}=\langle X_{k}^x-X_{k}^y,g(X_{k}^x)-g(X_{k}^y)\rangle\Delta B_{k}+(1-\theta)h\Delta B_{k}\langle f( X_{k}^x)-f(X_{k}^y),g(X_{k}^x)-g(X_{k}^y)\rangle.
\end{equation*}
It is not difficult to show that
\begin{align*}
|X_{k+1}^x-X_{k+1}^y|^2\leq&\frac{1+(1-\theta)^2h^2K_{1}+\Delta B_{k}^2K_{1}+2K_{2}(1-\theta)h}{1-2K_{2}\theta h}|X_{k}^x-X_{k}^y|^2+\frac{Q_1}{1-2K_{2}\theta h}.
\end{align*}
Since $\mathbb{E}\Delta B_{k}=0$, we have $\mathbb{E}Q_{1}=0$. Then, we obtain
\begin{align*}
\mathbb{E}|X_{k+1}^x-X_{k+1}^y|^2\leq \bar{C}_3\mathbb{E}|X_{k}^x-X_{k}^y|^2.
\end{align*}
By iteration, we have
\begin{equation*}
\mathbb{E}|X_{k+1}^x-X_{k+1}^y|^2\leq \bar{C}_3^{k+1}\mathbb{E}|x-y|^2,
\end{equation*}
where
\begin{equation*}
\bar{C}_3=\frac{1+(1-\theta)^2h^2K_{1}+hK_{1}+2K_{2}(1-\theta)h}{1-2K_{2}\theta h}.
\end{equation*}
Since $h<-(2K_2+K_1)/(1-\theta)^2K_1$, $1-2K_{2}\theta h>0$ and \eqref{k1k2}, we have
$0<\bar{C}_3<1$.
This complete the proof. \eproof
\end{proof}

\begin{lemma}
\label{lemma13}
Given Conditions \ref{xf} and \ref{ligrfg}, the solution generated by the ST method \eqref{STM} obeys
\begin{equation*}
\EE\left( \sup_{0 \leq k \leq n} |X_k|^2 \right)\leq C_{2},
\end{equation*}
where $C_{2}$ is a constant that can rely on k.
\end{lemma}
\begin{proof}
From \eqref{STM}, we have
\begin{align*}
|X_{k+1}|^{2}=&\langle X_{k}+(1-\theta) f(X_{k})h+g(X_{k})\Delta B_{k},X_{k+1} \rangle+\langle \theta f(X_{k})h, X_{k+1}\rangle\\
\leq&\frac{1}{2}|X_{k}+(1-\theta) f(X_{k})h +g(X_{k})\Delta B_{k}|^2 + \frac{1}{2}|X_{k+1}|^2 + \theta h(\mu|X_{k+1}|^2+a).
\end{align*}
Applying Conditions \ref{xf} and \ref{ligrfg}, we get
\begin{align*}
(1-\frac{1}{2}-\theta h \mu)|X_{k+1}|^2\leq &\frac{1}{2}[|X_{k}|^2+(1-\theta)^2 h^2 |f(X_{k})|^2+|g(X_{k})|^2|\Delta B_{k}|^2\\
&+ 2 h (1-\theta)\langle X_{k},f(X_{k})\rangle+2 \langle X_{k},g(X_{k})\rangle \Delta B_{k}\\
&+2 (1-\theta) h \langle f(X_{k}),g(X_{k})\rangle \Delta B_{k}]+\theta h a\\
\leq&\frac{1}{2}[|X_{k}|^2+(1-\theta)^2 h^2 (\kappa |X_{k}|^2+c)+|g(X_{k})|^2|\Delta B_{k}|^2\\
&+2 h(1-\theta)(\mu|X_{k}|^2 +a)+|X_{k}|^2+|g(X_{k})|^2|\Delta B_{k}|^2\\
&+\Delta(1-\theta)(|f(X_{k})|^2+|g(X_{k})|^2|\Delta B_{k}|^2)]+\theta h a\\
\leq&\frac{1}{2}[(2+(1-\theta)^2 h^2\kappa+2h(1-\theta)\mu+h(1-\theta)\kappa)|X_{k}|^2\\
&+(2+h(1-\theta))|g(X_{k})|^2|\Delta B_{k}|^2]+\frac{1}{2}h(1-\theta)c\\
&+\frac{1}{2}ch^2(1-\theta)^2+ah(1-\theta)+\theta h a.
\end{align*}
That is
\begin{align*}
|X_{k+1}|^2\leq D_{1}|X_{k}|^2+D_{2}|g(X_{k})|^2|\Delta B_{k}|^2+D_{3},
\end{align*}
where
\begin{align*}
D_{1}=&[1+\frac{1}{2}\kappa h^2(1-\theta)^2+h\mu(1-\theta)+\frac{1}{2}\kappa h(1-\theta)]/(\frac{1}{2}-h\mu\theta),\\
D_{2}=&[1+\frac{1}{2} h (1-\theta)]/(\frac{1}{2}-h \mu \theta),\\
D_{3}=&[\frac{1}{2}h(1-\theta)c+\frac{1}{2}ch^2(1-\theta)^2+ah(1-\theta)+ah\theta ]/(\frac{1}{2}-h\mu\theta).
\end{align*}
Summarizing both sides yields
\begin{align*}
\sum_{i=1}^{k+1}|X_{i}|^2=D_{1}\sum_{i=0}^{k}|X_{i}|^2+D_{2}\sum_{i=0}^{k}|g(X_{i})|^2|\Delta B_{k}|^2+(k+1)D_{3}.
\end{align*}
Now we have
\begin{align*}
|X_{k+1}|^2=(D_{1}-1)\sum_{i=0}^{k}|X_{i}|^2+|X_{0}|^2+D_{2}\sum_{i=0}^{k}|g(X_{i})|^2|\Delta B_{k}|^2+(k+1)D_{3}.
\end{align*}
Taking the supreme and expectation on both sides gives
\begin{align*}
\EE \left( \sup_{0 \leq k \leq n} |X_k|^2 \right)
\leq&(D_{1}-1)\sum_{i=0}^{k}\EE \left( \sup_{0 \leq k \leq n} |X_{i}|^2 \right)\\
&+D_{2}\EE \left( \sup_{0 \leq k \leq n}(\sum_{i=0}^{k}|g(X_{i})|^2|\Delta B_{k}|^2)\right)+(k+1)D_{3}+|X_{0}|^2.
\end{align*}
Then we have
\begin{align*}
\EE \left( \sup_{0 \leq k \leq n} |X_k|^2 \right)
\leq&(D_{1}-1+D_{2}h^2\sigma)\sum_{i=0}^{k}\EE \left( \sup_{0 \leq k \leq n} |X_k|^2 \right)\\
&+D_{2}h^2kb+(k+1)D_{3}+|X_{0}|^2,
\end{align*}
where $\EE |\D B_k|^2 = h$ is used. Using the discrete version of the Gronwall inequality, Lemma \ref{disGronine}, we have
\begin{equation*}
\EE \left( \sup_{0 \leq k \leq n} |X_k|^2 \right) \leq (D_{2}h^2kb+(k+1)D_{3}+|X_{0}|^2)\exp ((k+1)(D_{1}-1+D_{2}h^2\sigma)).
\end{equation*}
The proof is complete.    \eproof
\end{proof}
\par \noindent
Combining Lemmas \ref{lemma11}, \ref{lemma12} and \ref{lemma13} and using Chebyshev's inequality, we derive the existence and uniqueness of the stationary distribution of the ST method with $\theta \in [0, 1/2)$ from Theorem \ref{mainthm}.

\subsection{$\theta \in [1/2,1]$} \label{thetabiggerthanahalf}

When $\theta \in [1/2,1]$, we do not need the part for $f(x)$ in Condition \ref{ffgg} but only need that
\begin{equation}
\label{gggg}
 |g(x)-g(y)|^2\leq K_1 |x-y|^2,
\end{equation}
for any $x,y \in \RR^d$.

To prove Lemmas \ref{lemma21} and \ref{lemma22}, let us present the following two lemmas and we refer the readers to \cite{WWL2008} for the proof.
\begin{lemma}
\label{forlemma21}
Let Condition \ref{xf} hold, then for any $\beta_1$, $\beta_2\in\mathbb{R}$ with $\beta_2\geq\beta_1\geq0$, the inequality
\begin{equation*}
|x-\beta_1f(x)|^2+2\beta_1a\leq\frac{1-\mu\beta_1}{1-\mu\beta_2}(|x-\beta_2f(x)|^2+2\beta_2a)
\end{equation*}
holds.
\end{lemma}

\begin{lemma}
\label{forlemma22}
Let Condition \ref{xyff} hold, then for any $\lambda_1$, $\lambda_2\in\mathbb{R}$ with $\lambda_2\geq\lambda_1\geq0$, the inequality
\begin{equation*}
|x-y-\lambda_1[f(x)-f(y)]|\leq\frac{1-K_2\lambda_1}{1-K_2\lambda_2}|x-y-\lambda_2[f(x)-f(y)]|
\end{equation*}
holds.
\end{lemma}
Now we are ready to present the three main lemmas in this subsection.
\begin{lemma}
\label{lemma21}
Given Condition \ref{xf}, \eqref{gg} and \eqref{ms} hold, the solution generated by ST method \eqref{STM} obeys
 \begin{equation*}
 \EE|X_k|^{2}\leq c_1,
 \end{equation*}
 where $c_1$ is a constant that does not rely on k.
\end{lemma}
\begin{proof}
Denote
\begin{align*}
\theta^*=1+\frac{\sigma}{4\mu},~
\lambda=\frac{2\mu+\sigma}{2\mu}\wedge(2\theta-1).
\end{align*}
If $\theta\in[1/2,\theta^*]$, by the definition of $\lambda$ and Lemma \ref{forlemma21}, we have
\begin{align*}
&|X_{k+1}-(1-\theta+\lambda)hf(X_{k+1})|^2+2a(1-\theta+\lambda)h\\
\leq&|X_{k+1}-\theta f(X_{k+1})h+(2\theta-1-\lambda)f(X_{k+1})h|^2+2a(1-\theta+\lambda)h\\
\leq&|X_{k+1}-\theta f(X_{k+1})h|^2+2(2\theta-1-\lambda)\langle X_{k+1},f(X_{k+1})\rangle h\\
&+[(2\theta-1-\lambda)^2-\theta^2]|f(X_{k+1})|^2h^2+2a(1-\theta+\lambda)h\\
\leq&|X_k-(1-\theta)f(X_k)h|^2+4(1-\theta)\langle X_k,f(X_k)\rangle h+|g(X_k)|^2h\\
&+2(2\theta-1-\lambda)\langle X_{k+1},f(X_{k+1})\rangle h+2a(1-\theta+\lambda)h+G_k\\
\leq&|X_k-(1-\theta)f(X_k)h|^2+2(1-\theta)ah+4(1-\theta)\langle X_k,f(X_k)\rangle h+|g(X_k)|^2h\\
&+2(2\theta-1-\lambda)\langle X_{k+1},f(X_{k+1})\rangle h+2a(1-\theta+\lambda)h-2(1-\theta)ah+G_k\\
\leq&\frac{1-\mu(1-\theta)h}{1-\mu(1-\theta+\lambda)h}(|X_k-(1-\theta+\lambda)f(X_k)h|^2+2(1-\theta+\lambda)ah)\\
&+4(1-\theta)\langle X_k,f(X_k)\rangle h+|g(X_k)|^2h+2(2\theta-1-\lambda)\langle X_{k+1},f(X_{k+1})\rangle h\\
&+2a(1-\theta+\lambda)h-2(1-\theta)ah+G_k,
\end{align*}
where
\begin{equation*}
G_k=2\langle X_{k}+(1-\theta)hf(X_{k}),g(X_{k})\rangle\Delta B_k+g(X_{k})^2(\Delta B_k^2-h).
\end{equation*}
Denote
$$F_k=\mathbb{E}|X_k-(1-\theta+\lambda)f(X_k)h|^2+2(1-\theta+\lambda)ah,$$
and
$$N_h(\lambda)=[1-\mu(1-\theta)h]/[1-\mu(1-\theta+\lambda)h].$$
It is clear that $\mathbb{E}G_k=0$. By Condition \ref{xf} and \eqref{gg}, we have
\begin{align*}
F_{k+1}\leq& N_hF_k+[4(1-\theta)\mu h+\sigma h]\mathbb{E}|X_k|^2+2h\mu(2\theta-1-\lambda)\mathbb{E}|X_{k+1}|^2\\
&+(2a+b)h\\
\leq&N_h^{k+1}[F_0-2(2\theta-1-\lambda)\mu h|X_0|^2]+\psi_{\lambda}(h)h\sum\limits_{i=0}^{k}N_h^{k-i}\mathbb{E}|X_{i}|^2\\
&+2(2\theta-1-\lambda)\mu h\mathbb{E}|X_{k+1}|^2+\frac{1}{1-N_h}(2a+b)h,
\end{align*}
where
\begin{equation*}
\psi_{\lambda}(h)=4(1-\theta)+\sigma+2N_h\mu(2\theta-1-\lambda).
\end{equation*}
For $\theta\in(1/2,\theta^*]$, we have $\psi_{\lambda}(h)<0$, $2(2\theta-1-\lambda)\mu h\leq0$.\\
From Lemma \ref{forlemma21}, we have $\mathbb{E}|X_k|^2\leq F_k$ and $0<N_h(\lambda)<1$. Therefore, we get
\begin{align*}
\mathbb{E}|X_k|^2\leq& N_h^{k}[F_0-2(2\theta-1-\lambda)\mu h|X_0|^2]+\frac{1}{1-N_h}(2a+b)h\\
\leq&[F_0-2(2\theta-1-\lambda)\mu h|X_0|^2]+\frac{1}{1-N_h}(2a+b)h.
\end{align*}
Denote
\begin{equation*}
c_1=[F_0-2(2\theta-1-\lambda)\mu h|X_0|^2]+\frac{1}{1-N_h}(2a+b)h,
\end{equation*}
then, we have
\begin{equation*}
\EE|X_k|^{2}\leq c_1.
\end{equation*}
For $\theta\in(1+\sigma/(4\mu),1)$, choosing $\lambda'<\lambda$ sufficiently small such that $\psi_{\lambda'}(h)<0$ for any $h>0$, then using the same arguments above, we complete the proof.   \eproof
\end{proof}

\begin{lemma}
\label{lemma22}
Given Condition \ref{xyff} and \eqref{gggg} , for any two initial values $x,y\in\mathbb{R}_{d}$ with $x \neq y$ the solutions generated by the ST method \eqref{STM} obey
 \begin{equation*}
 \EE|X_{k}^x-X_{k}^y|^{2}\leq c_{3}
 \end{equation*}
with $\lim\limits_{i\rightarrow+\infty}c_{3}=0$.
\end{lemma}
\begin{proof}
Denote
\begin{align*}
\theta^*=1+\frac{K_1}{4K_2}, ~
\lambda=\frac{2K_2+K_1}{2K_2}\wedge(2\theta-1).
\end{align*}
If $\theta\in[1/2,\theta^*]$, by the definition of $\lambda$ and Lemma \ref{forlemma22}, we have
\begin{align*}
&\left||X_{k+1}^x-X_{k+1}^y|-(1-\theta+\lambda)h|f(X_{k+1}^x)-f(X_{k+1}^y)|\right|^2 \\
\leq &\left||X_{k+1}^x-X_{k+1}^y|-\theta h|f(X_{k+1}^x)-f(X_{k+1}^y)|\right|^2\\
&+2(2\theta-1-\lambda)h\langle X_{k+1}^x-X_{k+1}^y,f(X_{k+1}^x)-f(X_{k+1}^y)\rangle\\
\leq &\left||X_{k}^x-X_{k}^y|-(1-\theta)h|f(X_{k}^x)-f(X_{k}^y)|\right|^2\\
&+4(1-\theta)h\langle X_{k}^x-X_{k}^y,f(X_{k}^x)-f(X_{k}^y)\rangle\\
&+|g(X_{k}^x)-g(X_{k}^y)|^2h+2(2\theta-1-\lambda)h\langle X_{k+1}^x-X_{k+1}^y,f(X_{k+1}^x)-f(X_{k+1}^y)\rangle+M_k\\
\leq&\left|\frac{1-K_2(1-\theta)h}{1-K_2(1-\theta+\lambda)h}\right|^2\left||X_{k}^x-X_{k}^y|-(1-\theta+\lambda)h|f(X_{k}^x)-f(X_{k}^y)|\right|^2\\
&+4(1-\theta)h\langle X_{k}^x-X_{k}^y,f(X_{k}^x)-f(X_{k}^y)\rangle+|g(X_{k}^x)-g(X_{k}^y)|^2h\\
&+2(2\theta-1-\lambda)h\langle X_{k+1}^x-X_{k+1}^y,f(X_{k+1}^x)-f(X_{k+1}^y)\rangle+M_k,
\end{align*}
where
\begin{align*}
M_k=&2\langle|X_{k}^x-X_{k}^y|+(1-\theta)h|f(X_{k}^x)-f(X_{k}^y)|,|g(X_{k}^x)-g(X_{k}^y)|\rangle\Delta B_k\\
&+|g(X_{k}^x)-g(X_{k}^y)|^2(|\Delta B_k^2-h).
\end{align*}
Denote
$$W_k=\mathbb{E}\left||X_{k}^x-X_{k}^y|-(1-\theta)h|f(X_{k}^x)-f(X_{k}^y)|\right|^2,$$
and
$$L_h(\lambda)=|(1-K_2(1-\theta)h)/(1-K_2(1-\theta+\lambda)h)|^2.$$
It is not hard to see that $\mathbb{E}M_k=0$. By Condition \ref{xyff} and \eqref{gggg}, we have
\begin{align*}
W_{k+1}\leq& L_h W_k+[4(1-\theta)K_2h+K_1 h]\mathbb{E}|X_{k}^x-X_{k}^y|^2+2K_2 h(2\theta-1-\lambda)\mathbb{E}|X_{k+1}^x-X_{k+1}^y|^2\\
\leq&L_h^{k+1}[A_0-2(2\theta-1-\lambda)K_2 h\mathbb{E}|x-y|^2]+\varphi_{\lambda}(h)h\sum\limits_{i=0}^{k}L_h^{k-i}\mathbb{E}|X_{i}^x-X_{i}^y|^2,
\end{align*}
where $\varphi_{\lambda}^{h}=4(1-\theta)K_2+K_1+2(2\theta-1-\lambda)K_2L_h$.\\
For $\theta\in[1/2,\theta^*]$, we have $\varphi_{\lambda}^{h}<0$, $\mathbb{E}|X_k|^2\leq W_k$, and $|L_h|<1$. Then, we get
\begin{equation*}
 \mathbb{E}|X_{k}^x-X_{k}^y|^2\leq L_h^k[W_0-2(2\theta-1-\lambda)K_2h|x-y|^2].
\end{equation*}
For $\theta\in(1+K_1/(4K_2),1)$, choosing $\lambda'<\lambda$ sufficiently small such that $\psi_{\lambda'}(h)<0$ for any $h>0$ and using the same arguments above, we complete the proof. \eproof
\end{proof}

\begin{lemma} \label{lemma23}
Assume that Conditions \ref{xf}, \eqref{gg} and \eqref{ms} hold, then
\begin{equation*}
\EE\left( \sup_{0 \leq k \leq n} |X_k|^2 \right)\leq c_{2}
\end{equation*}
where $c_{2}$ is a constant that can rely on k.
\end{lemma}
\begin{proof}
From (\ref{STM}), we have
\begin{equation*}
|X_{k+1} - \theta f(X_{k+1}) h |^2 = |X_k - \theta f(X_k)h + f(X_k)h + g(X_k) \Delta B_k|^2.
\end{equation*}
Rewriting the right hand side, we have
\begin{align*}
|X_{k+1} - \theta f(X_{k+1}) h |^2
&= |X_{k} - \theta f(X_{k}) h |^2 + 2 \langle X_k, f(X_k)h \rangle + (1 - 2 \theta) |f(X_k)|^2 h^2 \\
&~~+ |g(X_k)\Delta B_k|^2 + \frac{2}{\theta} \langle X_k, g(X_k)\Delta B_k \rangle \\
&~~- \frac{2(1 - \theta)}{\theta} \langle X_k - \theta f(X_k)h, g(X_k)\Delta B_k \rangle.
\end{align*}
Due to the fact that $\theta \in [1/2,1]$ and Condition \ref{monocond}, we have
\begin{align*}
|X_{k+1} - \theta f(X_{k+1}) h |^2
&\leq |X_{k} - \theta f(X_{k}) h |^2 + |g(X_k)\Delta B_k|^2 + 2(\mu |X_k|^2 + a)h \\
&~~+ \frac{2}{\theta} \langle X_k, g(X_k)\Delta B_k \rangle +  \frac{2(1 - \theta)}{\theta} \langle X_k - \theta f(X_k)h, g(X_k)\Delta B_k \rangle.
\end{align*}
Summarising both sides yields
\begin{align}\label{bESup}
|X_{k+1} - \theta f(X_{k+1}) h |^2
&\leq |X_0 - \theta f(X_0) h |^2 + \sum_{i=0}^k |g(X_i) \Delta B_i|^2 + 2ah(k+1) + 2\mu h \sum_{i=0}^k |X_i|^2 \nonumber \\
&~~+ \frac{2}{\theta} \sum_{i=0}^k \langle X_k, g(X_k)\Delta B_k \rangle +  \frac{2(1 - \theta)}{\theta} \sum_{i=0}^k \langle X_k - \theta f(X_k)h, g(X_k)\Delta B_k \rangle.
\end{align}
By the elementary inequality, it is not hard to see that
\begin{align} \label{ESupdB1}
&~~~~\EE \left( \sup_{0 \leq k \leq n} \left| \sum_{i=0}^k \langle X_i - \theta f(X_i)h, g(X_i)\Delta B_i \rangle \right| \right)\nonumber \\
&\leq \EE \left( \sum_{i=0}^n \bigg|  \langle X_i - \theta f(X_i)h, g(X_i)\Delta B_i \rangle \bigg| \right) \nonumber \\
&\leq \frac{1}{2} \sum_{i=0}^n \EE \left| X_i - \theta f(X_i)h \right|^2 + \frac{1}{2} \sum_{i=0}^n \EE \left| g(X_i) \Delta B_i \right|^2.
\end{align}
Using (\ref{gg}) and the fact that $\EE|\Delta B_i|^2 = h$, we have
\begin{equation}\label{ESupg}
\EE \left( \sup_{0 \leq k \leq n} \sum_{i=0}^k |g(X_i) \Delta B_i|^2 \right) \leq \sum_{i=0}^n \EE (\sigma |X_i|^2 + b)h.
\end{equation}
Applying the elementary inequality and (\ref{ESupg}), we have
\begin{align}\label{Esupxgw}
\EE \left( \sup_{0 \leq k \leq n} \sum_{i=0}^k \langle X_i, g(X_i)\Delta B_i \rangle \right)
&\leq \sum_{i=0}^n \EE \left| \langle X_i, g(X_i)\Delta B_i \rangle  \right| \nonumber \\
&\leq \frac{1}{2}  \sum_{i=0}^n \EE |X_i|^2 + \leq \frac{1}{2}  \sum_{i=0}^n \EE (\sigma |X_i|^2 + b)h.
\end{align}
Now, taking the expectation and the supreme on both sides of (\ref{bESup}) and using (\ref{ESupdB1}), (\ref{ESupg}), and (\ref{Esupxgw}) we have
\begin{align}
\label{beforegronwall}
\EE \left( \sup_{0 \leq k \leq n} |X_{k+1} - \theta f(X_{k+1}) h |^2 \right) &\leq \EE \left| X_0 - \theta f(X_0) h \right|^2 + K \sum_{i=0}^n \EE |X_k|^2 \nonumber \\
&~~+\frac{1}{2} \sum_{i=0}^n \EE \left| X_i - \theta f(X_i)h \right|^2 + Kh,
\end{align}
where K is a generic constant.
Due to Condition \ref{monocond}, we have
\begin{align}
\label{yandyminus}
|X_{k} - \theta f(X_{k}) h |^2 &= |X_{k}|^2 - 2 \theta h \langle X_k, f(X_k) \rangle + \theta^2 h^2 |f(X_{k})|^2 \nonumber \\
&\geq |X_{k}|^2  - 2 \theta h \left( - \mu |X_{k}|^2 - a   \right) + \theta^2 h^2 |f(X_{k})|^2  \nonumber \\
&\geq (1 + 2 \theta h \mu) |X_{k}|^2 + 2 a \theta h.
\end{align}
Applying the discrete version of the Gronwall inequality and (\ref{yandyminus}) to (\ref{beforegronwall}), the assertion holds. \eproof
\end{proof}
\par \noindent
Combining Lemmas \ref{lemma21}, \ref{lemma22} and \ref{lemma23} and using Chebyshev's inequality, we derive the existence and uniqueness of the stationary distribution of the ST method with $\theta \in [1/2, 1]$ from Theorem \ref{mainthm}.

\subsection{The Convergence}
\label{theconvergencesec}
\label{conver}
Given Conditions \ref{ffgg} to \ref{ligrfg}, the convergence of the numerical stationary distribution to the underlying stationary distribution is discussed in this subsection.
\par
Recall that the probability measure induced by the numerical solution, $X_k$, is denoted by $\PP_k(\cdot,\cdot)$, similarly we denote the probability measure induced by the underlying solution,$x(t)$, by $\bar{\PP}_t(\cdot,\cdot)$.

\begin{lemma}
\label{numundlem}
Let Conditions \ref{ffgg} to \ref{ligrfg} hold and fix any initial value $x_0 \in \RR^d$. Then, for any given $T_1>0$ and $\e >0$ there exists a sufficiently small $\D t^* >0$ such that
\begin{equation*}
d_{\LL}(\bar{\PP}_{k \D t}(x_0,\cdot),\PP_k(x_0,\cdot)) < \e
\end{equation*}
provided that $\D t < \D t^*$ and $k \D t \leq T_1$.
\end{lemma}
The result can be derived from the finite time strong convergence of the ST method \cite{ZWH2015}.
\par
Now we are ready to show that the numerical stationary distribution converges to the underlying stationary distribution as time step diminishes.
\begin{theorem}
Given Conditions \ref{ffgg} to \ref{ligrfg}, then
\begin{equation*}
\lim_{\D t \rightarrow 0} d_{\LL} (\Pi_{\D t}(\cdot), \pi(\cdot)) = 0.
\end{equation*}
\end{theorem}
\begin{proof}
Fix any initial value $x_0 \in \RR^d$ and set $\e >0$ to be an arbitrary real number. Due to the existence and uniqueness of the stationary distribution of the underlying equation, there exists a $\Theta^* >0$ such that for any $t > \Theta^*$
\begin{equation*}
d_{\LL} (\bar{\PP}_t(x_0,\cdot),\pi(\cdot)) < \e/3.
\end{equation*}
Similarly, by Theorem \ref{mainthm}, there exists a pair of $\D t^{**}>0$ and $\Theta^{**} >0$ such that
\begin{equation*}
d_{\LL} (\PP_k(x_0,\cdot),\Pi_{\D t}(\cdot)) < \e/3
\end{equation*}
for all $\D t < \D t^{**}$ and $k \D t > \Theta^{**}$.
Let $\Theta = \max(\Theta^*,\Theta^{**})$, from Lemma \ref{numundlem} there exists a $\D t^*$ such that for any $\D t < \D t^*$ and $k \D t < \Theta + 1$
\begin{equation*}
d_{\LL} (\bar{\PP}_{k \D t}(x_0,\cdot),\PP_k(x_0,\cdot)) < \e/3.
\end{equation*}
Therefore, for any $\D t < \min(\D t^*,\D t^{**})$, set $k = [\Theta /\D t] + 1/\D t$, we see the assertion holds by the triangle inequality. \eproof
\end{proof}

\section{Simulations}\label{secsimu}

We present three numerical results in this section to demonstrate the theoretical results. The first one is a linear scale SDE with the true stationary distribution known. The second one is also a scale SDE but with the super-linear drift coefficient, of which the stationary distribution can be found by solving some ordinary differential equation. The third one is a two dimensional case.

\begin{expl}
\begin{equation}
\label{1dlexpl}
dx(t)=-\alpha x(t)dt+\sigma dB(t)\quad on\quad ~~~t\geq0.
\end{equation}
\end{expl}
Given any initial value $ X_{0}=x(0)\in \RR$, from (\ref{STM}) we have
\begin{equation*}
X_{k+1}=X_{k}-\alpha\theta X_{k+1} h -(1-\theta)\alpha X_{k} h+\sigma \Delta B_{k}.
\end{equation*}
This gives that $X_{k+1}$ is normally distributed with mean
\begin{equation*}
E(X_{k+1})=(1-\frac{\alpha h}{1+\alpha \theta h})^{k+1} x(0).
\end{equation*}
And the variance is
\begin{align*}
Var(X_{k+1})
&=\frac{(1+\theta \alpha h)^{-2}}{(1+\theta \alpha h)^{2}}(1-\alpha h+\theta \alpha h)^{2}Var(X_{k})+\sigma^{2}h(1+\theta \alpha h)^{-2}\\
&=\sigma^{2}h[(1+\theta \alpha h)^{-2}+(1+\theta \alpha h)^{-4}(1-\alpha h+\theta \alpha h)^{2}\\
&+(1+\theta \alpha h)^{-6}(1-\alpha h+\theta \alpha h)^{4}+\cdots +(1+\theta \alpha h)^{-2(k+1)}(1-\alpha h+\theta \alpha h)^{2k}]\\
&=\frac{1-[(1+\theta \alpha h)^{-2(k+1)}(1-\alpha h+\theta \alpha h)^{2(k+1)}]}{(1+\alpha \theta h)^{2}-((1-\alpha h+\theta \alpha h)^{2}}\\
&=\frac{\sigma^{2}}{2\alpha-\alpha^{2}h+2\alpha^{2}\theta h}.
\end{align*}
So the distribution of the solution generated by the ST method approaches the normal distribution $N(0,\frac{\sigma^{2}}{2\alpha-\alpha^{2}h+2\alpha^{2}\theta h})$ as $k\rightarrow\infty$.
\par
Choosing $\alpha = \sigma =2$, we draw several pictures. In this setting, the true stationary distribution is the standard normal distribution.
\par

Figure \ref{1DLinDes} shows the empirical density function of the numerical solution to (\ref{1dlexpl}). Here the step size is chosen to be 0.001, the terminal time is 10 and the initial value is 2. 1000 sample paths with the $\theta = 1/2$ are used to draw the graph. It can be seen that with the time advancing the density function is tending to an stable one, which indicates the existence of the stationary distribution.

\begin{figure}[htb]
\centering $
\begin{array}{c}
\includegraphics[width=5in]{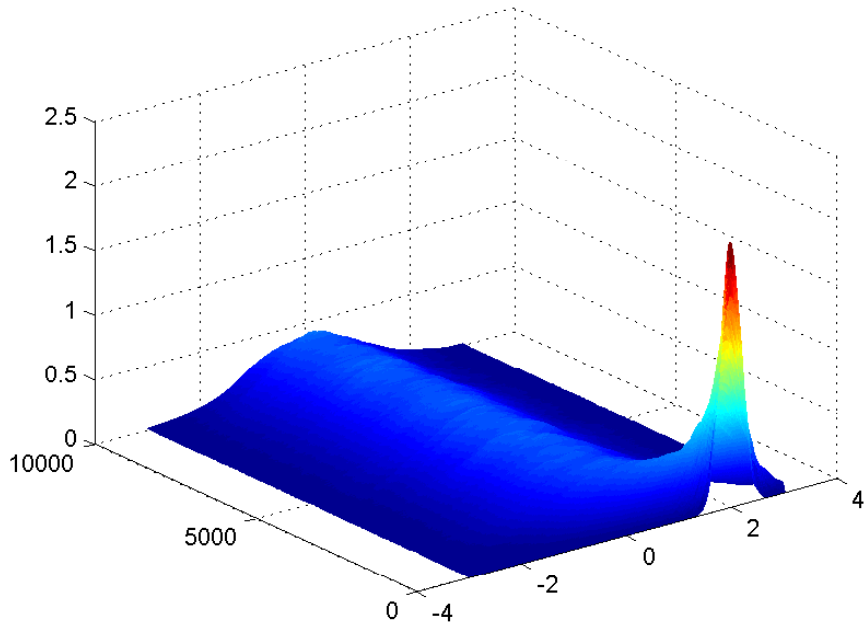}
\end{array}$
\caption{The empirical density function along the number of iterations }
\label{1DLinDes}
\end{figure}

Now, we use the Kolmogorov-Smirnov test (K-S test) \cite{M1951a} to measure the difference between the numerical stationary distribution and the true stationary distribution. Figure \ref{1DLinKS} displays the changes in the p value as the time advances. It can be seen that after roughly t=1.8, the K-S test indicates that one can not reject that the samples generated by the numerical method are from the true distribution with the $95\%$ confidence.

\begin{figure}[h]
\centering$
\begin{array}{c}
\includegraphics[width=5in]{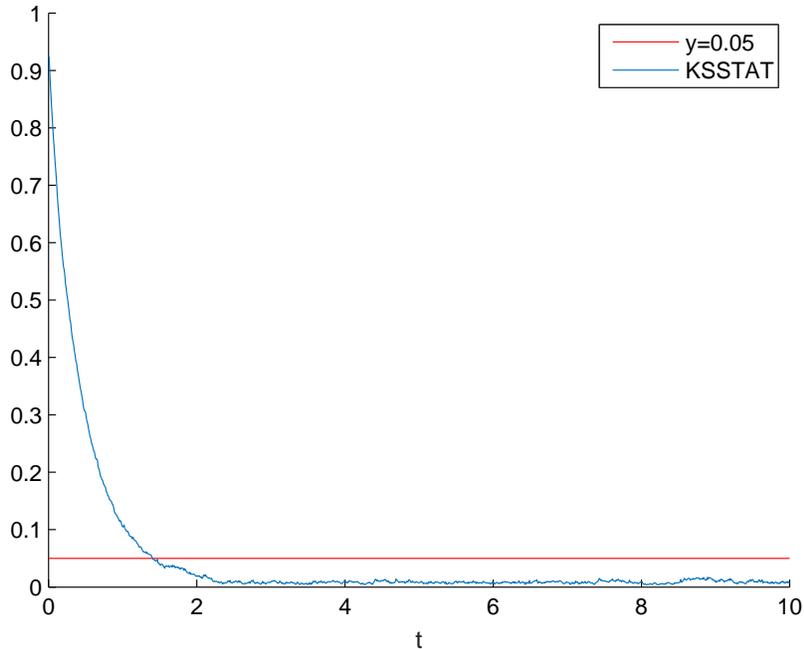}
\end{array}$
\caption{The p values along the time line}
\label{1DLinKS}
\end{figure}

The convergent rate of the numerical stationary distribution with different choices of $\theta$ to the true stationary distribution is plotted in Figure \ref{1DLinrate}. The step sizes, $2^{-1},~2^{-2},~2^{-3},~2^{-4}$, are used at $T=10$. It can be seen that the convergent rate is approximately one.

\begin{figure}[h]
\begin{center}$
\begin{array}{cc}
\includegraphics[width=3in]{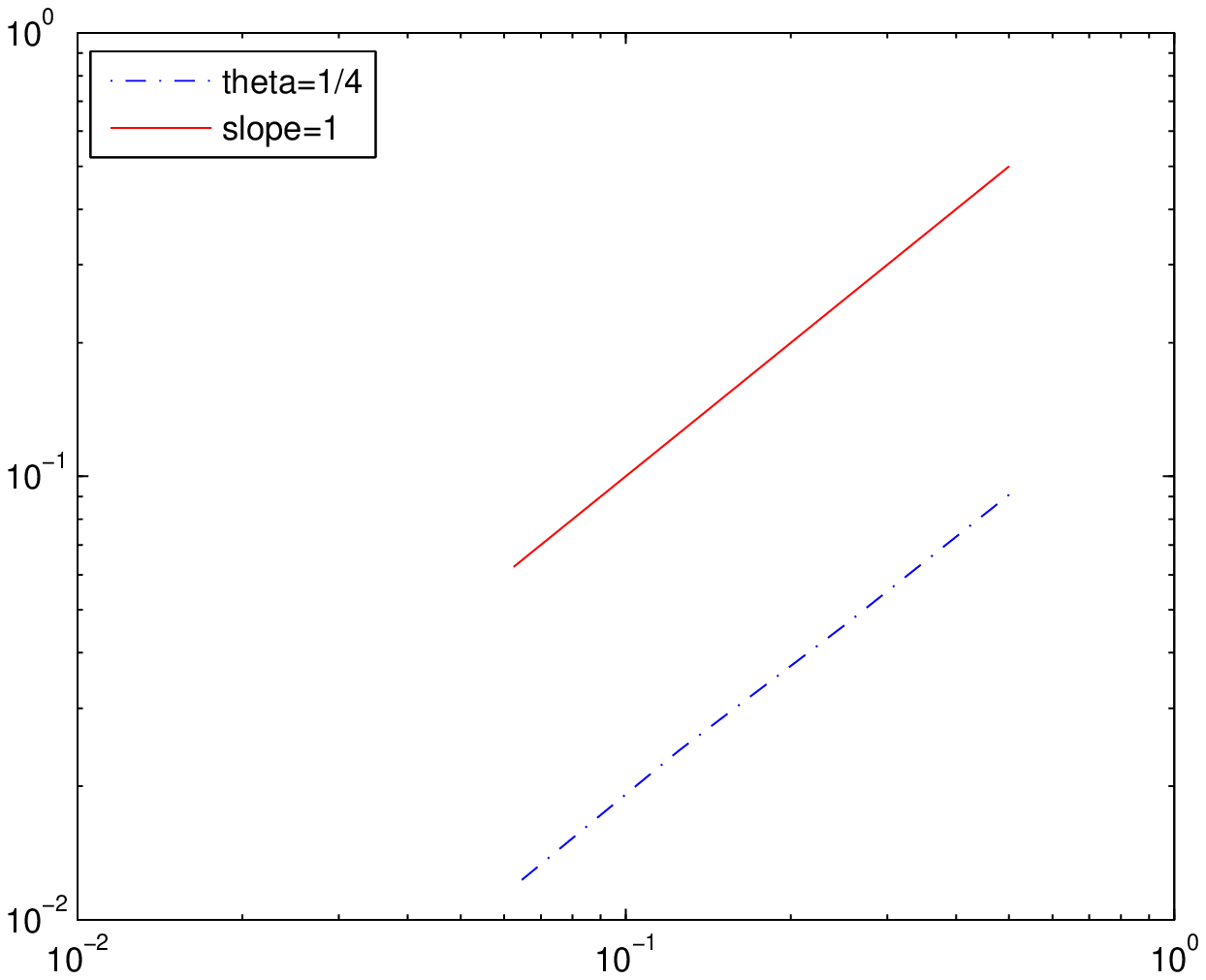}&
\includegraphics[width=3in]{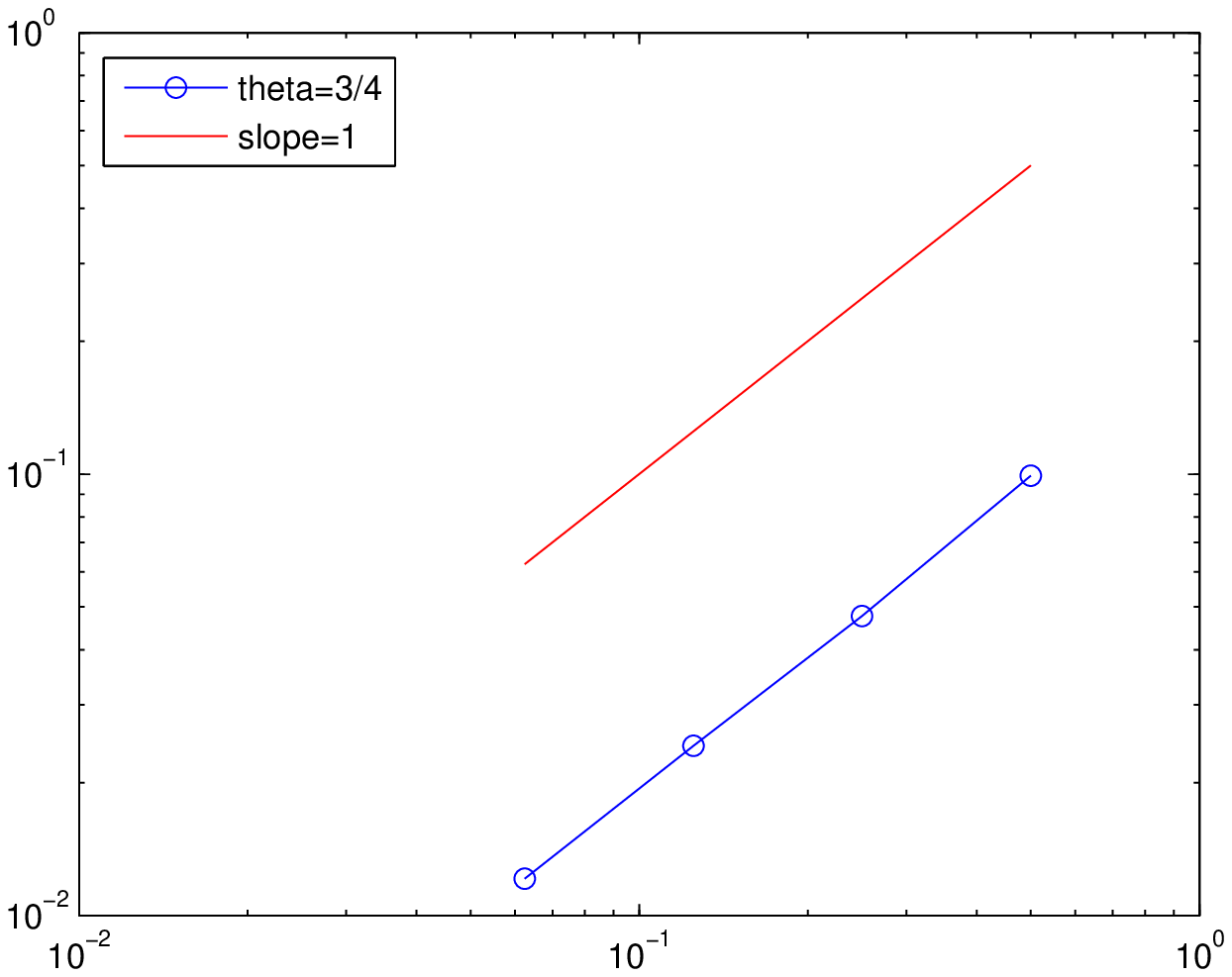}
\end{array}$
\end{center}
\caption{Convergent rate of the linear case}
\label{1DLinrate}
\end{figure}

Next, we consider the SDE with the super-linear drift coefficient.
\begin{expl}
\label{splod}
\begin{equation*}
 dx(t)=-0.5(x(t)+x^{3}(t))dt+dB(t),
\end{equation*}
with $x(0)=x_0$.
\end{expl}

The corresponding Kolmogorov-Fokker-Planck equation for the theoretical probability density function of the stationary distribution $p(x)$ is
\begin{equation*}
 0.5\frac{d^2 p(x)}{dx^2} - \frac{d}{dx} (-0.5(x+x^3)p(x)) = 0.
\end{equation*}
And the exact solution is known to be \cite{SOO1973}
\begin{equation*}
 p(x) = \frac{1}{I_{\frac{1}{4}}(\frac{1}{8})+I_{-\frac{1}{4}}(\frac{1}{8})} \exp(\frac{1}{8}-\frac{1}{2} x^2 - \frac{1}{4}x^4),
\end{equation*}
where $I_{\nu}(x)$ is a modified Bessel function of the first kind.
\par
Figure \ref{1Dsp3d} shows the changes of the  empirical density function with the time advancing. It can be seen that with the time variable increasing the center of the density function rapidly moves from the initial value, 2, to the theoretical centre one. And the density function is quite stable as time goes large.

\begin{figure}[htb]
\begin{center}$
\begin{array}{c}
\includegraphics[width=5in]{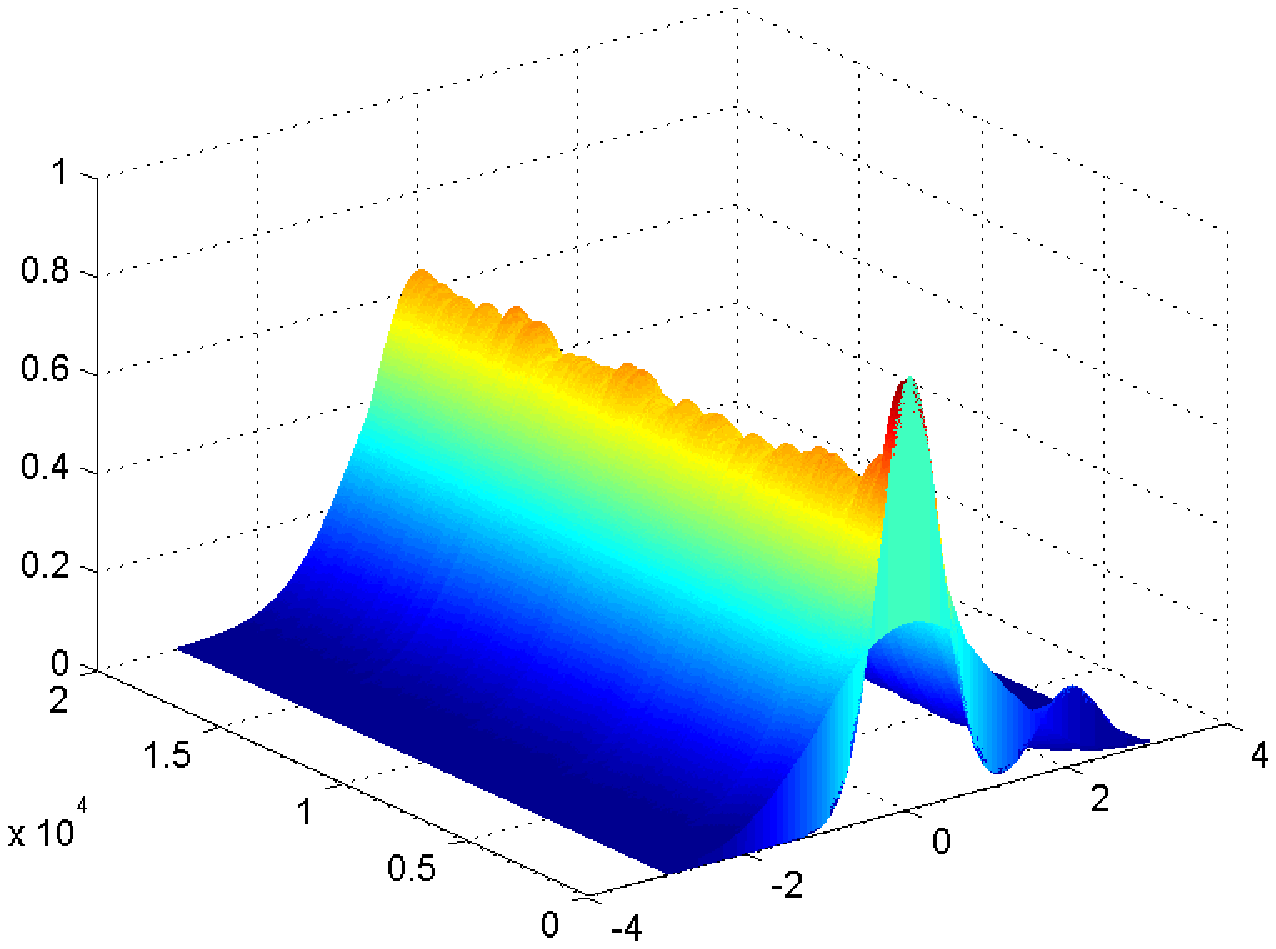}
\end{array}$
\end{center}
\caption{The empirical density function along the time line}
\label{1Dsp3d}
\end{figure}

Figure \ref{1Dsprate} shows the convergent rate at $T=10$ with the step sizes, $2^{-1},~2^{-2},~2^{-3},~2^{-4}$. It can be seen that the rate for the super-linear case is not as good as the linear case, but the plots still show the convergence as the step size getting small.

\begin{figure}[htb]
\begin{center}$
\begin{array}{cc}
\includegraphics[width=3in]{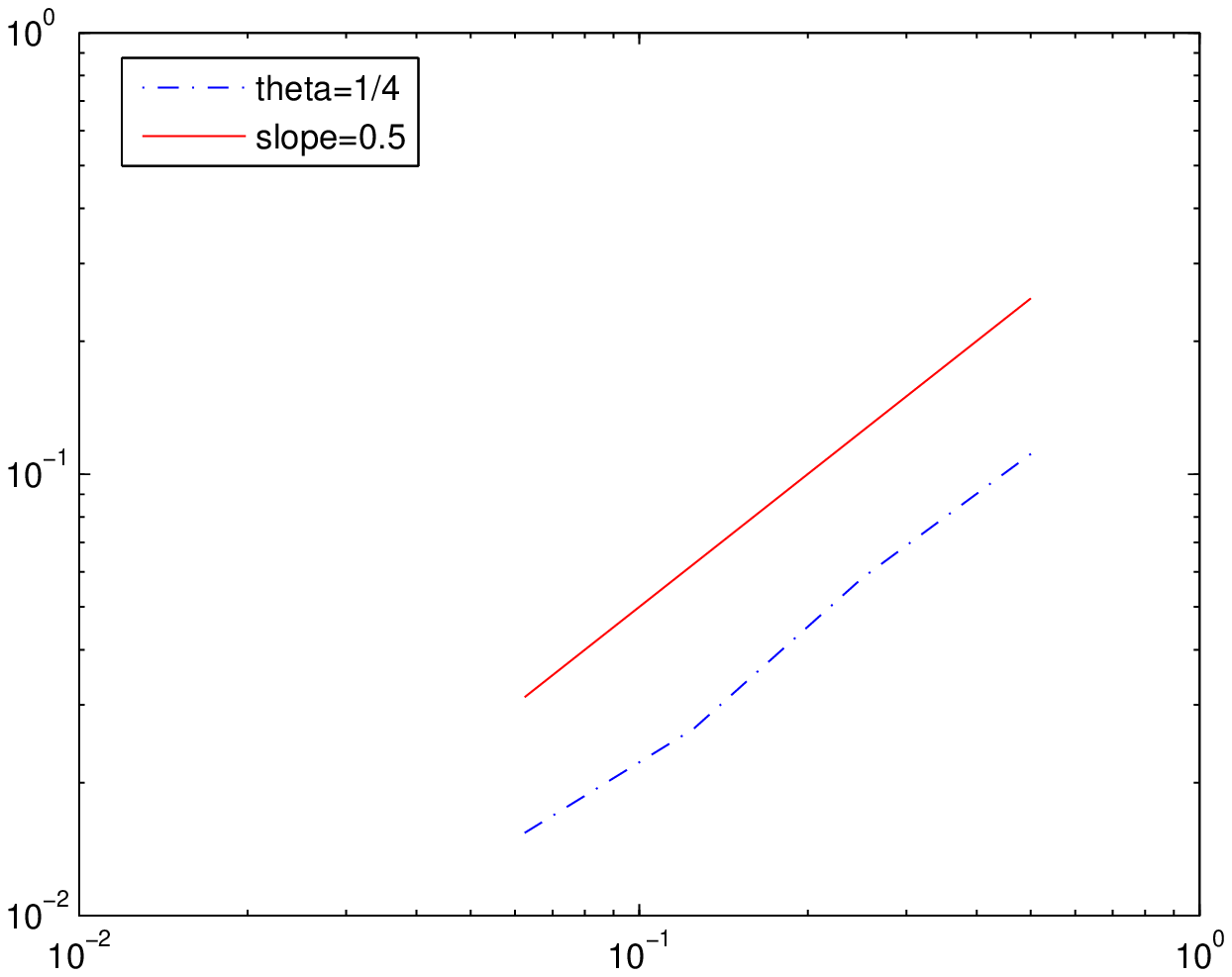}&
\includegraphics[width=3in]{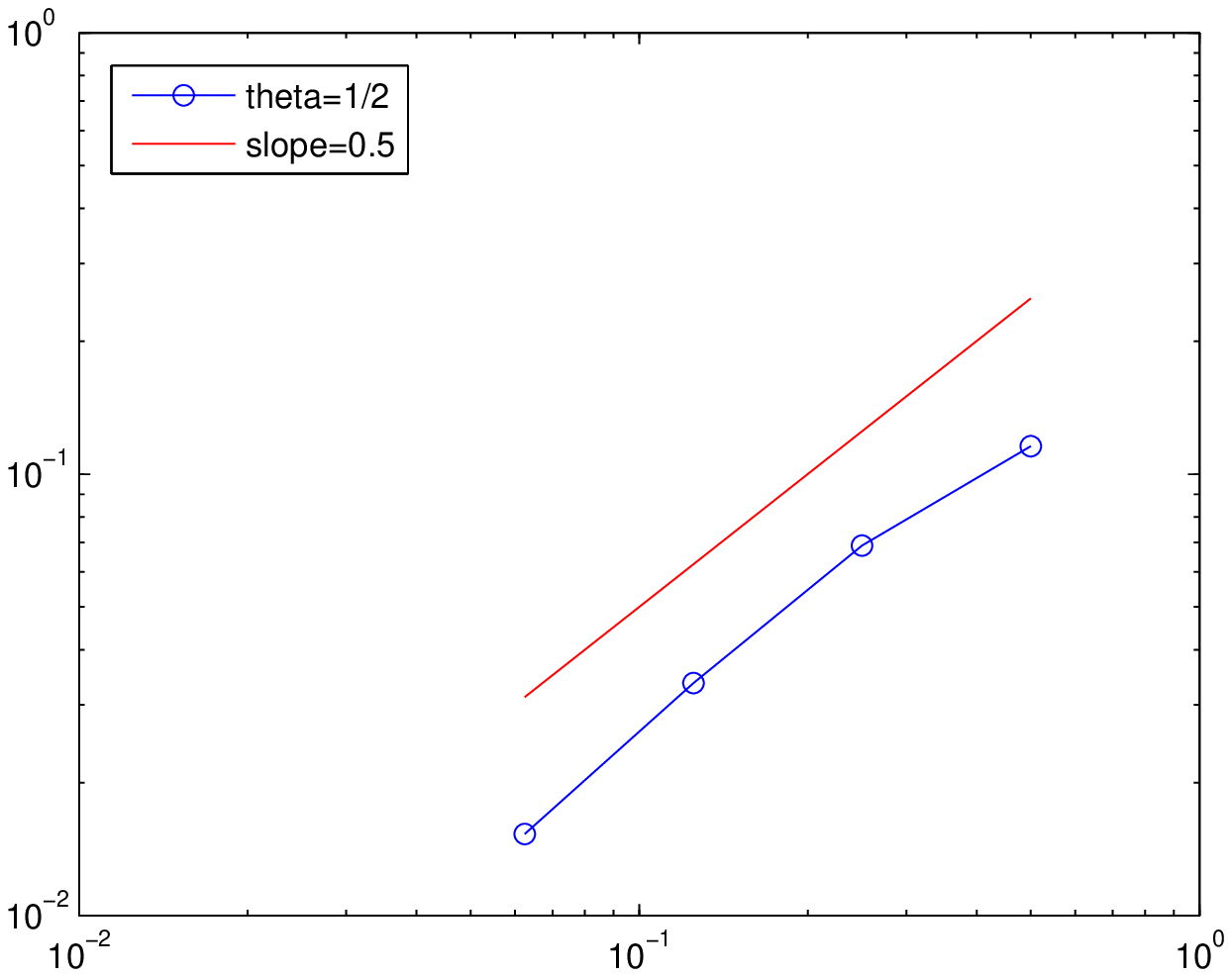}
\end{array}$
\end{center}
\caption{Convergent rate of the super-linear case}
\label{1Dsprate}
\end{figure}

Figure \ref{1DspKS} shows the change of the p value in the K-S test with the time increasing. It can be seen that after the time $t=3$ approximately one can not reject the numerical samples are from the true distribution with $95\%$ confidence. This is a little bit worse than the linear case, as one need to wait a bit longer to see the stationary distribution.

\begin{figure}[htb]
\begin{center}$
\begin{array}{c}
\includegraphics[width=5in]{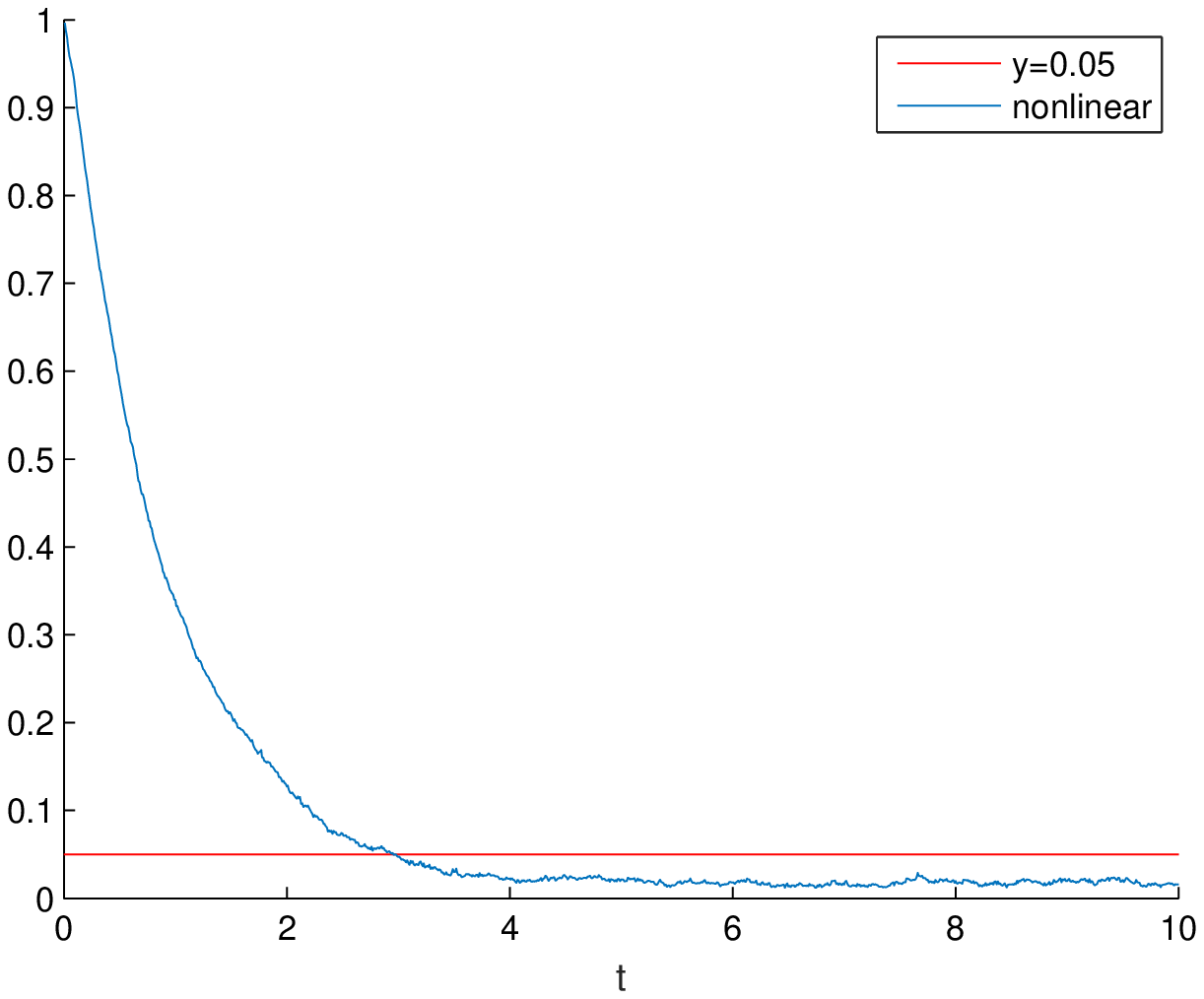}
\end{array}$
\end{center}
\caption{The p value along the time line}
\label{1DspKS}
\end{figure}

We turn to the two dimensional super-linear SDEs.
\begin{expl}
\begin{equation*}
d\left[
\begin{matrix}
x_{1}(t)\\
x_{2}(t)
\end{matrix}
\right]
=\left[
\begin{matrix}
-x_{1}^{3}(t)-5x_{1}(t)+x_{2}(t)+5\\
-x_{2}^{3}(t)-x_{1}(t)-5x_{2}(t)+5
\end{matrix}
\right]dt+\left[
\begin{matrix}
x_{1}(t)-x_{2}(t)+3\\
-x_{1}(t)-x_{2}(t)+3
\end{matrix}
\right]dB(t),
\end{equation*}
with some initial data.
\end{expl}

We check Condition \ref{xyff}, \eqref{k1k2} and \eqref{gggg}. Using the fundamental inequality that $ab \leq 1/2(a^2 + b^2)$, we have
\begin{align*}
\langle x-y,&f(x) - f(y)\rangle\\
&=(x_{1}-y_{1},x_{2}-y_{2})\left(
\begin{matrix}
f_{1}(x)-f_{1}(y)\\
f_{2}(x)-f_{2}(y)
\end{matrix}
\right)\\
&=(x_{1}(t)-y_{1}(t),x_{2}(t)-y_{2}(t))\left(
\begin{matrix}
-x_{1}^{3}(t)-5x_{1}(t)+x_{2}(t)+y_{1}^{3}(t)+5y_{1}(t)-y_{2}(t)\\
-x_{2}^{3}(t)-5x_{2}(t)-x_{1}(t)+y_{2}^{3}(t)+5y_{2}(t)+y_{1}(t)
\end{matrix}
\right)\\
&=((-x_{1}^{3}(t)+y_{1}^{3}(t))+5(y_{1}(t)-x_{1}(t))+(x_{2}(t)-y_{2}(t)))(x_{1}(t)-y_{1}(t))\\
&\qquad +((-x_{2}^{3}(t)+y_{2}^{3}(t))+5(y_{2}(t)-x_{2}(t))+(-x_{1}(t)+y_{1}(t)))(x_{2}(t)-y_{2}(t))\\
&=(-x_{1}^{3}(t)+y_{1}^{3}(t))(x_{1}(t)-y_{1}(t))+(-x_{2}^{3}(t)+y_{2}^{3}(t))(x_{2}(t)-y_{2}(t))\\
&\qquad -5((x_{1}(t)-y_{1}(t))^{2}+(x_{2}(t)-y_{2}(t))^{2})\\
&=-(x_{1}(t)-y_{1}(t))^{2}(y_{1}^{2}(t)+x_{1}(t)y_{1}(t)+x_{1}^{2}(t))-(x_{2}(t)-y_{2}(t))^{2}(y_{2}^{2}(t)+x_{2}(t)y_{2}(t)+x_{2}^{2}(t))\\
&\qquad -5((x_{1}(t)-y_{1}(t))^{2}+(x_{2}(t)-y_{2}(t))^{2}).
\end{align*}
Since
\begin{align*}
y_{1}^{2}(t)+x_{1}(t)y_{1}(t)+x_{1}^{2}(t)
\geq\frac{y_{1}^{2}(t)+2x_{1}(t)y_{1}(t)+x_{1}^{2}(t)}{2}
\geq\frac{(x_{1}(t)+y_{1}(t))^2}{2}
\geq0,
\end{align*}
we have
\begin{align*}
\langle(x-y),(f(x) - f(y)) \rangle &\leq ((x_{1}(t)-y_{1}(t))^{2}+(x_{2}(t)-y_{2}(t))^{2})-5((x_{1}(t)-y_{1}(t))^{2}+(x_{2}(t)-y_{2}(t))^{2})\\
&=-4((x_{1}(t)-y_{1}(t))^{2}+(x_{2}(t)-y_{2}(t))^{2}).
\end{align*}
Also, we have
\begin{align*}
|g(x)-g(y)|^{2}&=((x_{1}(t)-y_{1}(t))-(x_{2}(t)-y_{2}(t)))^{2}+((y_{2}(t)-x_{2}(t))-(x_{1}(t)-y_{1}(t)))^{2}\\
&=2((x_{1}(t)-y_{1}(t))^{2}+(x_{2}(t)-y_{2}(t))^{2}),
\end{align*}
and
\begin{equation*}
2\langle(x-y),(f(x) - f(y)) \rangle + |g(x)-g(y)|^{2}\leq -6((x_{1}(t)-y_{1}(t))^{2}+(x_{2}(t)-y_{2}(t))^{2}).
\end{equation*}

We plot the two dimensional empirical density function at different time. The initial values are $[2,3]^T$, the step size is 0.1 and $2 \times 10^6$ sample points are used to draw the plots. We can see from Figure \ref{2DsT} that when the time is small the density function changes quite a lot even within a small time interval. But with time goes by, the density function stabilise to some certain shape, which could be regarded as the stationary distribution. Figure \ref{2DlT} shows almost no difference between the empirical density functions at $T=18$ and $T=20$.

\begin{figure}[htb]
\begin{center}$
\begin{array}{cc}
\includegraphics[width=3in]{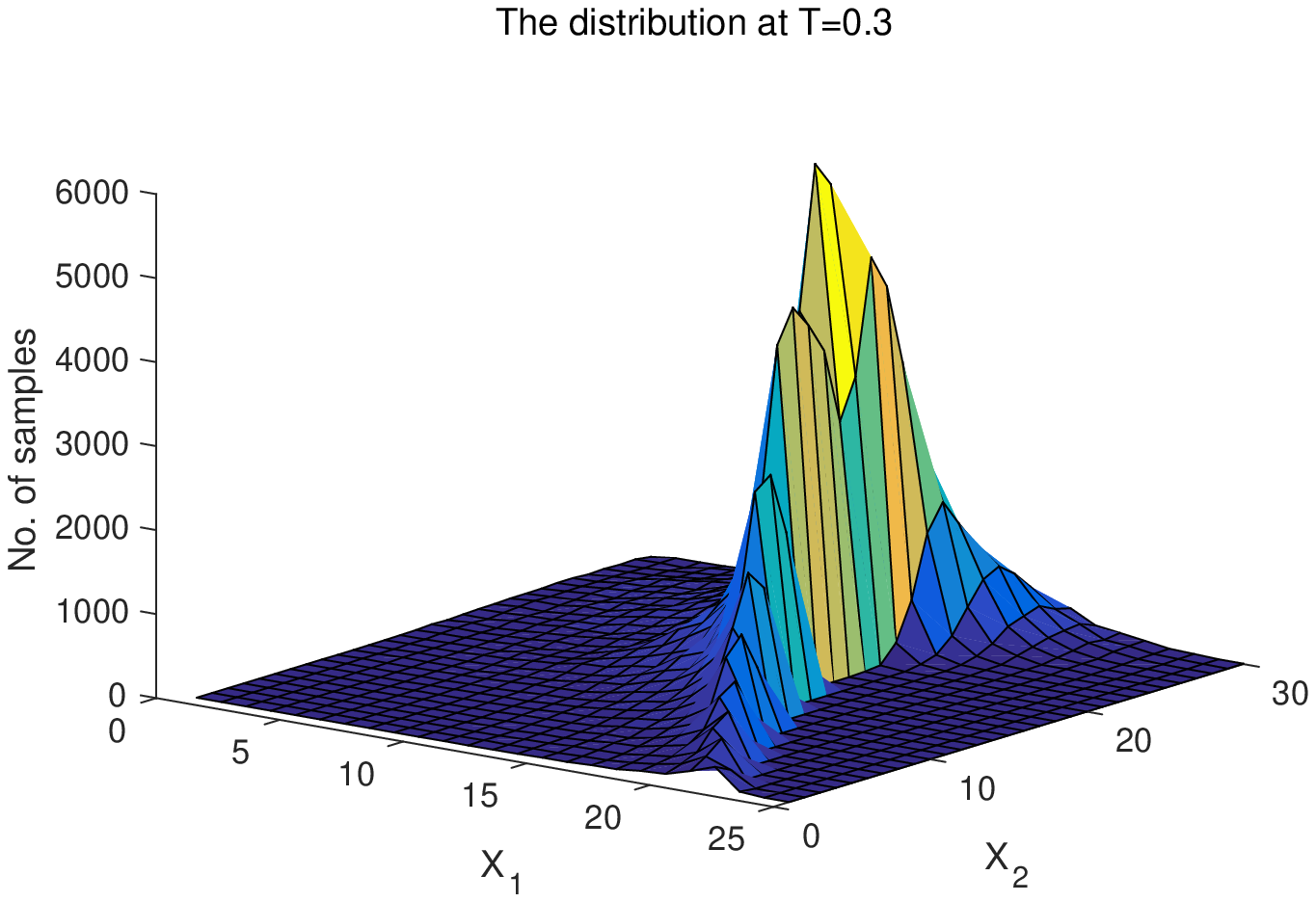}&
\includegraphics[width=3in]{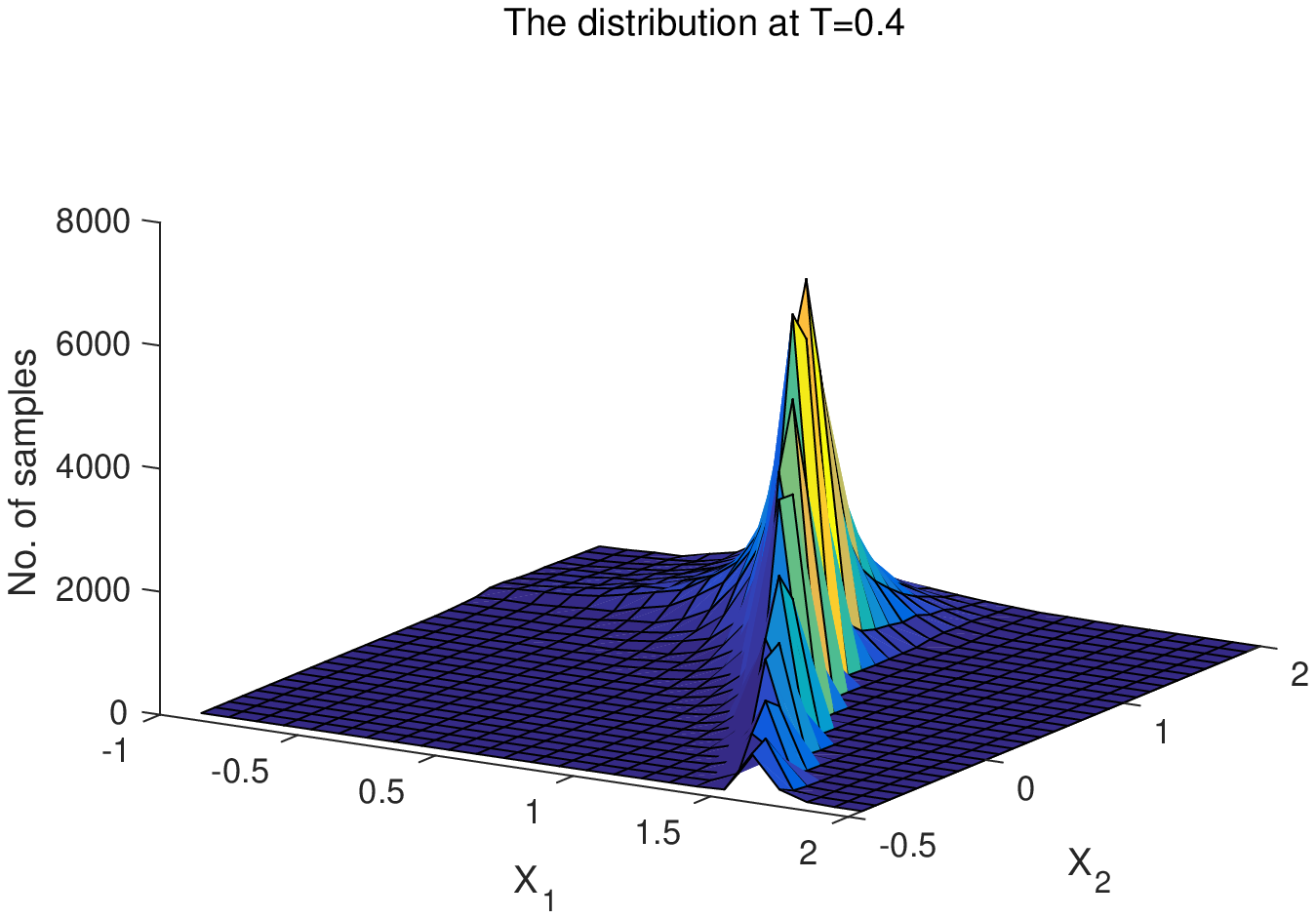}
\end{array}$
\end{center}
\caption{The empirical density function at the small time}
\label{2DsT}
\end{figure}

\begin{figure}[htb]
\begin{center}$
\begin{array}{cc}
\includegraphics[width=3in]{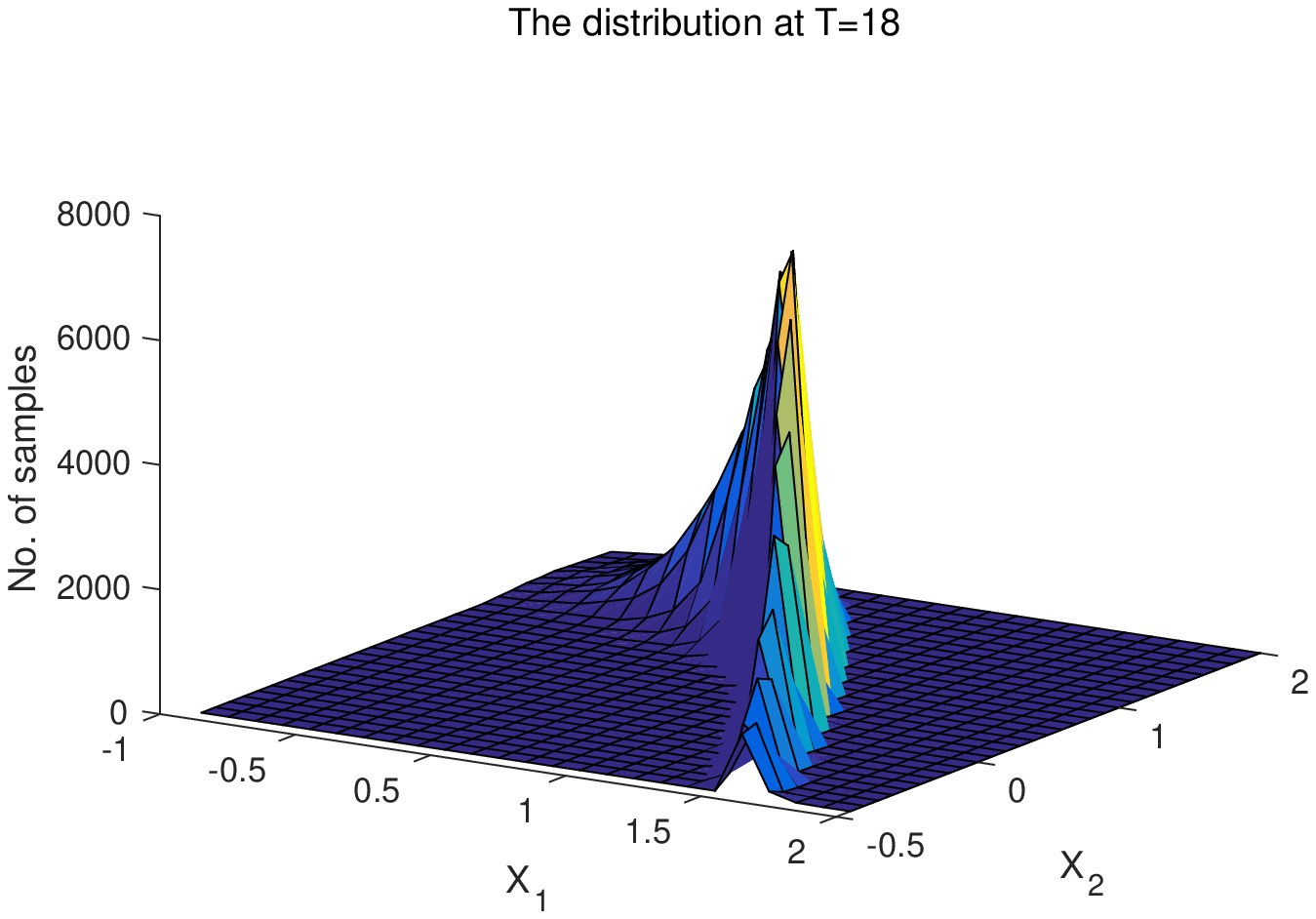}&
\includegraphics[width=3in]{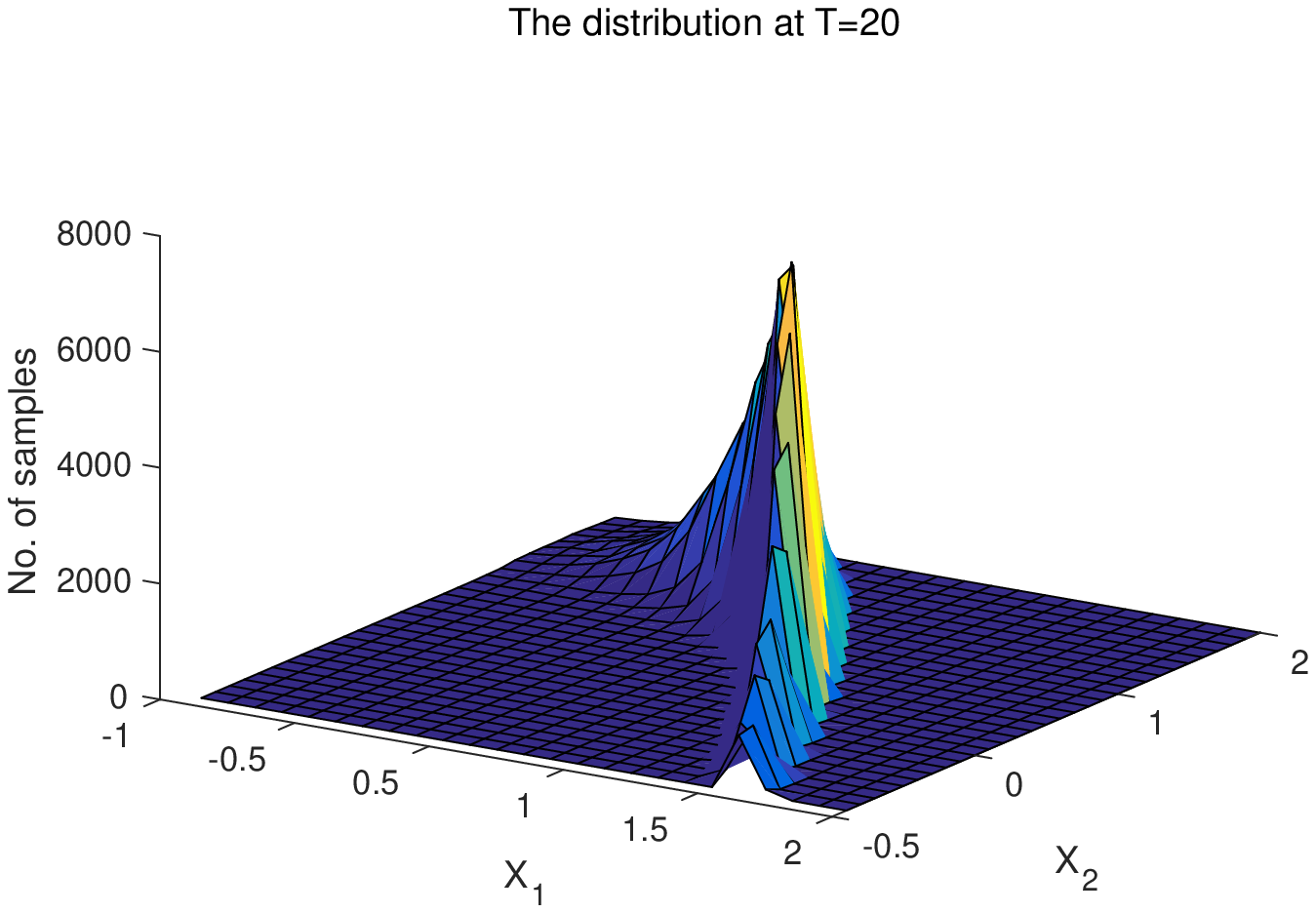}
\end{array}$
\end{center}
\caption{The empirical density function at the relatively large time}
\label{2DlT}
\end{figure}

\section{Conclusion} \label{secconclu}
In this paper, we study the numerical stationary distributions generated by the stochastic theta methods. Both the drift and diffusion coefficients are required to satisfy the global Lipschitz condition when $\theta \in [0, 1/2)$, but some super-linear terms are allowed to appear in the drift coefficient when $\theta \in [1/2, 1]$. Three numerical examples are given to show that the convergence and convergent rate of the numerical stationary distributions to their true counterparts. The plots also indicate that the numerical stationary distributions from the numerical solutions to SDEs could be used to approximate some non-linear deterministic differential equations.

%\section*{Acknowledgement}
%The authors would like to thank
%the National Natural Science Foundation of China (11701378),
%``Chenguang Program'' supported by Shanghai Education Development Foundation and Shanghai Municipal Education Commission (16CG50)
%and
%Shanghai Pujiang Program (16PJ1408000),
%for their financial support.


\clearpage


\section*{References}


\begin{thebibliography}{10}

\bibitem{abdulle2017optimal}
A.~Abdulle, I.~Almuslimani, and G.~Vilmart.
\newblock Optimal explicit stabilized integrator of weak order one for stiff
  and ergodic stochastic differential equations.
\newblock {\em arXiv preprint arXiv:1708.08145}, 2017.

\bibitem{BSY2016}
J.~Bao, J.~Shao, and C.~Yuan.
\newblock Approximation of invariant measures for regime-switching diffusions.
\newblock {\em Potential Anal.}, 44(4):707--727, 2016.

\bibitem{BBKR2012a}
G.~Berkolaiko, E.~Buckwar, C.~Kelly, and A.~Rodkina.
\newblock Almost sure asymptotic stability analysis of the
  {$\theta$}-{M}aruyama method applied to a test system with stabilising and
  destabilising stochastic perturbations.
\newblock {\em LMS J. Comput. Math.}, 15:71--83, 2012.

\bibitem{CW2012}
L.~Chen and F.~Wu.
\newblock Almost sure exponential stability of the {$\theta$}-method for
  stochastic differential equations.
\newblock {\em Statist. Probab. Lett.}, 82(9):1669--1676, 2012.

\bibitem{Hig2000a}
D.~J. Higham.
\newblock Mean-square and asymptotic stability of the stochastic theta method.
\newblock {\em SIAM J. Numer. Anal.}, 38(3):753--769 (electronic), 2000.

\bibitem{HJK2011a}
M.~Hutzenthaler, A.~Jentzen, and P.~E. Kloeden.
\newblock Strong and weak divergence in finite time of {E}uler's method for
  stochastic differential equations with non-globally {L}ipschitz continuous
  coefficients.
\newblock {\em Proc. R. Soc. Lond. Ser. A Math. Phys. Eng. Sci.},
  467(2130):1563--1576, 2011.

\bibitem{IW1981a}
N.~Ikeda and S.~Watanabe.
\newblock {\em Stochastic differential equations and diffusion processes},
  volume~24 of {\em North-Holland Mathematical Library}.
\newblock North-Holland Publishing Co., Amsterdam, 1981.

\bibitem{LM2015}
W.~Liu and X.~Mao.
\newblock Numerical stationary distribution and its convergence for nonlinear
  stochastic differential equations.
\newblock {\em J. Comput. Appl. Math.}, 276:16--29, 2015.

\bibitem{M1991a}
X.~Mao.
\newblock {\em Stability of stochastic differential equations with respect to
  semimartingales}, volume 251 of {\em Pitman Research Notes in Mathematics
  Series}.
\newblock Longman Scientific \& Technical, Harlow, 1991.

\bibitem{M2008a}
X.~Mao.
\newblock {\em Stochastic differential equations and applications}.
\newblock Horwood Publishing Limited, Chichester, second edition, 2008.

\bibitem{MS2013a}
X.~Mao and L.~Szpruch.
\newblock Strong convergence rates for backward {E}uler-{M}aruyama method for
  non-linear dissipative-type stochastic differential equations with
  super-linear diffusion coefficients.
\newblock {\em Stochastics}, 85(1):144--171, 2013.

\bibitem{MYY2005a}
X.~Mao, C.~Yuan, and G.~Yin.
\newblock Numerical method for stationary distribution of stochastic
  differential equations with {M}arkovian switching.
\newblock {\em J. Comput. Appl. Math.}, 174(1):1--27, 2005.

\bibitem{M1951a}
F.~J. Massey~J.
\newblock The kolmogorov-smirnov test for goodness of fit.
\newblock {\em Journal of the American statistical Association},
  46(253):68--78, 1951.

\bibitem{QLH2014}
Q.~Qiu, W.~Liu, and L.~Hu.
\newblock Asymptotic moment boundedness of the stochastic theta method and its
  application for stochastic differential equations.
\newblock {\em Adv. Difference Equ.}, pages 2014:310, 14, 2014.

\bibitem{Rod05}
A.~Rodkina and H.~Schurz.
\newblock Almost sure asymptotic stability of drift-implicit {$\theta$}-methods
  for bilinear ordinary stochastic differential equations in {$R^1$}.
\newblock {\em J. Comput. Appl. Math.}, 180(1):13--31, 2005.

\bibitem{SOO1973}
T.~T. Soong.
\newblock {\em Random differential equations in science and engineering}.
\newblock Academic Press [Harcourt Brace Jovanovich Publishers], New York,
  1973.
\newblock Mathematics in Science and Engineering, Vol. 103.

\bibitem{talay1990second}
D.~Talay.
\newblock Second-order discretization schemes of stochastic differential
  systems for the computation of the invariant law.
\newblock {\em Stochastics: An International Journal of Probability and
  Stochastic Processes}, 29(1):13--36, 1990.

\bibitem{WWL2008}
W.~Wang, L.~Wen, and S.~Li.
\newblock Nonlinear stability of {$\theta$}-methods for neutral differential
  equations in {B}anach space.
\newblock {\em Appl. Math. Comput.}, 2008.

\bibitem{YM2003a}
C.~Yuan and X.~Mao.
\newblock Asymptotic stability in distribution of stochastic differential
  equations with {M}arkovian switching.
\newblock {\em Stochastic Process. Appl.}, 103(2):277--291, 2003.

\bibitem{YM2004a}
C.~Yuan and X.~Mao.
\newblock Stability in distribution of numerical solutions for stochastic
  differential equations.
\newblock {\em Stochastic Anal. Appl.}, 22(5):1133--1150, 2004.

\bibitem{YM2005a}
C.~Yuan and X.~Mao.
\newblock Stationary distributions of {E}uler-{M}aruyama-type stochastic
  difference equations with {M}arkovian switching and their convergence.
\newblock {\em J. Difference Equ. Appl.}, 11(1):29--48, 2005.

\bibitem{ZW2014}
X.~Zong and F.~Wu.
\newblock Choice of {$\theta$} and mean-square exponential stability in the
  stochastic theta method of stochastic differential equations.
\newblock {\em J. Comput. Appl. Math.}, 255:837--847, 2014.

\bibitem{ZWH2015}
X.~Zong, F.~Wu, and C.~Huang.
\newblock Theta schemes for {SDDE}s with non-globally {L}ipschitz continuous
  coefficients.
\newblock {\em J. Comput. Appl. Math.}, 278:258--277, 2015.

\end{thebibliography}
\end{document}